\documentclass{article}
\usepackage[a4paper, top=2cm,bottom=2.5cm,left=2cm,right=2cm,marginparwidth=1.5cm]{geometry}
\usepackage{amsthm}

\usepackage{graphicx}
\usepackage[english]{babel}
\usepackage[utf8]{inputenc}
\usepackage{authblk,amsmath,amsfonts,amssymb}
\usepackage{mathtools}
\usepackage{tikz, tikz-cd}
\usetikzlibrary{patterns}
\usepackage{stmaryrd}
\usepackage{color}
\usepackage[shortlabels]{enumitem}
\usepackage{hyperref}
\usepackage{nameref}
\usepackage{todonotes}
\usepackage{soul,placeins}
\usepackage[font={small},labelfont={bf}]{caption}
\usepackage{hhline}
\usepackage{multirow}
\usepackage{dcolumn}

\usepackage{biblatex}
\usepackage[normalem]{ulem}
\newcommand{\stkout}[1]{\ifmmode\text{\sout{\ensuremath{#1}}}\else\sout{#1}\fi}

\newcommand{\cI}{\mathcal{I}}
\newcommand{\IE}{\mathbb{I}(E)}
\newcommand{\GammaInt}{\Gamma^{\text{i}}_h}
\newcommand{\GammaExt}{\Gamma^{\text{b}}_h}

\newcommand{\Mesh}{\mathcal{M}_h}

\newcommand{\dspace}{\mathcal{V}^r_h}

\newcommand{\avg}[1]{\left\lbrace\mskip-5mu\lbrace{#1}\right\rbrace\mskip-5mu\rbrace}

\newcommand{\binplus}{{} + }
\newcommand{\binminus}{{} - }
\newcommand{\phantomplus}{\: \phantom{+} \:}
\newcommand{\phantomminus}{\: \phantom{-} \:}
\newcommand{\phantomeq}{\mathrel{\phantom{=}}}
\newcommand{\trace}[1]{\operatorname{tr}_{\gamma}(#1)}
\newcommand{\extop}[1]{\mathcal{L}_{#1}}

\newcommand{\capE}{\mathfrak{c}_E}

\usepackage{accents}
\newlength{\dhatheight}
\newcommand{\dhat}[1]{%
    \settoheight{\dhatheight}{\ensuremath{\hat{#1}}}%
    \addtolength{\dhatheight}{-0.35ex}%
    \hat{\vphantom{\rule{1pt}{\dhatheight}}%
    \smash{\hat{#1}}}}

\newtheorem{l2-dissipation-prop}{Proposition}
\newtheorem{l2-stability}[l2-dissipation-prop]{Corollary}

\theoremstyle{definition}
\newtheorem{remark}{Remark}
\newtheorem{definition}{Definition}
\newtheorem{example}{Example}
\theoremstyle{definition}
\newtheorem{lemma}{Lemma}
\newtheorem{proposition}{Proposition}
\newtheorem{theorem}{Theorem}

\title%[An Energy-Preserving DoD Stabilization for the Linear Wave Equation]
{An Energy-Preserving Domain of Dependence Stabilization for the Linear Wave Equation on Cut-Cell Meshes}
\author{Gunnar Birke, Christian Engwer, Sandra May, Louis Petri, Hendrik Ranocha}
\date{}

\begin{document}

\newpage

\maketitle

\begin{abstract}
We present an energy-preserving (either energy-conservative or energy-dissipative) domain of dependence stabilization method for the linear wave equation on cut-cell meshes. Our scheme is based on a standard discontinuous Galerkin discretization in space and an explicit (strong stability preserving) Runge Kutta method in time. Tailored stabilization terms allow for selecting the time step length based on the size of the background cells rather than the small cut cells by propagating information across small cut cells. The stabilization terms preserve the energy stability or energy conservation property of the underlying discontinuous Galerkin space discretization. Numerical results display the high accuracy and stability properties of our scheme.
\end{abstract}

\section{Introduction}

Cut-cell meshes have been the focus of considerable research over the years due to their easy generation compared to more traditional meshing methods. Objects or interfaces are placed inside a so-called background mesh, often Cartesian, and their intersections with elements from the background mesh are removed. Although this process can be implemented very efficiently, it gives up nearly any control over the shape and sizes of mesh elements, resulting potentially in arbitrarily small cut cells. In the context of solving first-order hyperbolic conservation laws, this leads to the \textit{small cell problem}: severe and practically infeasible time step size restrictions when explicit time stepping methods are employed.

Over the last few decades a variety of different approaches have emerged to handle the \textit{small cell problem}. Perhaps the conceptually easiest method is so called cell merging/agglomeration where small cut cells are merged with their bigger neighbors, forming a new grid without small cut cells; see, e.g.,  \cite{Bayyuk_Powell_vanLeer, Quirk1994} and also \cite{QIN201324, Oberlack2016} for more recent developments.
Although cell merging is conceptually simple, a general implementation is very involved and requires working directly on the mesh. 

The alternative line of attack is a suitable modification of the numerical scheme in the neighborhood of small cut cells. Here, several solution approaches for solving hyperbolic conservation laws have been developed in the context of finite volume (FV) and discontinuous Galerkin (DG) schemes. Examples for FV schemes are the \textit{h-box}-method~\cite{Berger_Helzel_Leveque_2003, Berger_Helzel_2012}, flux redistribution~\cite{Chern_Colella, Colella2006}, dimensionally split methods~\cite{Klein_cutcell, Gokhale_Nikiforakis_Klein_2018}, state redistribution (SRD) \cite{srd_fv, srd_dg}, and mixed explicit-implicit methods \cite{May_Berger_explimpl}. To our knowledge only few approaches exist for higher-order DG finite element methods discretizations. Specifically we want to mention the extension of SRD to DG~\cite{srd_dg}, the Domain of Dependence (DoD) stabilization~\cite{DoD_2d_linadv_2020, DoD_1d_nonlin_2022} and adaptations of the ghost penalty stabilization hyperbolic problems~\cite{Fu_Kreiss_2021, Sticko_Kreiss_2016_Nitsche},
originally developed for elliptic equations~\cite{Burman_Hansbo_2012}.

Despite this active development, the \textit{small cell problem} still poses a significant challenge w.r.t. the aim of explicit, higher-order methods with provable  stability and conservative guarantees. Regarding the analysis of higher-order cut-cell methods for the wave equation, we are aware of two approaches in the context of DG schemes. Sticko and Kreiss \cite{Sticko_Kreiss_2016_Nitsche,Sticko_Kreiss_2019_higherorder_wave} proposed a ghost penalty type stabilization for the second-order formulation and proved energy stability. Chan and co-authors \cite{srd_wave_eq_energy_stable} obtained results for the SRD method applied to the wave equation in first-order formulation. Both approaches rely on a stabilization that is not based on physical properties of the PDE operator; as a consequence the authors are able to prove energy stability, but energy conservation is violated.

Our approach for designing an energy-stable and energy-conservative scheme
that solves the \textit{small cell problem} 
is based on 
the DoD method, first introduced in \cite{DoD_2d_linadv_2020} for the linear advection equation.
The idea is to augment a standard DG scheme by spatial stabilization terms on small cut cells to ensure stability of the time discretization under reasonable time step sizes.
The design of the stabilization terms depends on the problem that needs to be solved.
The DoD stabilization was extended to nonlinear equations in 1d in \cite{DoD_1d_nonlin_2022}, and as 
a first-order version
to linear systems in 2d in \cite{fvca_proc_dod}. 
Both works are not straightforward to generalize; \cite{DoD_1d_nonlin_2022} makes significant use of the 1d structure, whereas the construction in \cite{fvca_proc_dod} is only feasible for for piecewise constant polynomials. Regarding the analysis of the DoD method, fully discrete stability results were obtained for the linear advection equation; in 2d for piecewise constant polynomials \cite{error-estimate-dod} and in 1d for arbitrary order polynomials \cite{petri2025domain}. Using the stability proof, \cite{error-estimate-dod} further derives an a-priori error estimate in 2d.  Additionally in \cite{petri2026kinetic} a SBP property was shown for the DoD stabilization, again in 1d and for the linear advection equation. This work uses a splitting between central and dissipative fluxes and we will apply a very similar idea here.

\par The main contribution of this paper is a generalization of the original DoD stabilization \cite{DoD_2d_linadv_2020} which allows us to derive cut-cell stabilizations for the linear wave equation, including an appropriate handling of reflecting boundary conditions.
The generalized construction is based on a distinction between central fluxes and dissipative fluxes, which are handled separately. To stabilize high-order discretizations, we introduce the concept of, as we call them, \emph{propagation forms}, generalizing the initial extension of fluxes. These propagation forms need to fulfill a particular integration-by-parts style structure. Using this abstract property, we can prove energy stability and even derive an energy-conservative semi-discretization in space, by controlling the added dissipation. A technical challenge was the handling of reflecting boundary conditions, which is again simplified by the distinction of central fluxes and dissipative fluxes and the construction of suitable propagation forms.

This article is structured as follows: In Section~\ref{sec: base setting} we introduce the linear wave equation as our model equation as well as the standard DG scheme on which our stabilized scheme is based upon. 
Section~\ref{sec: dod section} contains the description of the generalized DoD stabilization itself and the particular construction for the wave equation.
Our main theoretical result is presented in Section~\ref{sec: theoretical section} where we prove that our stabilized scheme is energy-stable/conservative, i.e., it preserves the energy stability/conservation property of the underlying DG scheme. 
In Section~\ref{sec: numerical section} we present numerical results, displaying the accuracy and long term stability properties of our scheme. We also provide a short discussion on numerical precision issues.

\section{General setting} \label{sec: base setting}

We consider the acoustic wave equation written in its first-order form. Let
\begin{align}
\label{eq:state-vector}
u &: \Omega \times [0,T] \rightarrow \mathbb{R}^3,\quad
u(x,t) = \begin{pmatrix}
    p(x, t)\\v_1(x, t)\\v_2(x, t)
\end{pmatrix},\quad \Omega \subset \mathbb{R}^2
\end{align}
{be the state vector with pressure $p$ and velocity components $v = (v_1, v_2)^T$. In the following we will drop the explicit dependence on $x$ and $t$. Let}
\begin{align}
\label{eq:physical-flux}
f(u) &= \begin{pmatrix}f_1(u)\\f_2(u)\end{pmatrix} 
= \begin{pmatrix}A_1 u\\ A_2 u\end{pmatrix},\quad
A_1 = {\footnotesize\left(\begin{array}{@{}c@{~\,}c@{~\,}c@{}}
        0 & c & 0\\
        c & 0 & 0\\
        0 & 0 & 0
    \end{array}\right)}, \;
A_2 = {\footnotesize\left(\begin{array}{@{}c@{~\,}c@{~\,}c@{}}
        0 & 0 & c\\
        0 & 0 & 0\\
        c & 0 & 0
    \end{array}\right)},% \quad c > 0 \text{ constant} \Sandra{??},
\end{align}
be the flux function with $c > 0$ a constant denoting the speed of sound. Given a unit vector $n$ the flux in direction $n$ is given as $f_n = f \cdot n = n_1 f_1 + n_2 f_2$ and can be computed as $f_n(u) = A_n u$ with
\begin{equation} \label{eq: directional flux matrix}
A_n = A \cdot n = n_1 A_1 + n_2 A_2.
\end{equation}
We point out that $A_n$ is symmetric and has two nonzero eigenvalues $c$ and $-c$.

The wave equation, written as a first-order system of hyperbolic conservation laws, takes the form
\begin{subequations}
\begin{alignat}{3}
    \partial_t u + \nabla \cdot f(u) & = 0 & \quad & \text{in } \Omega \times (0,T),\\ %\Omega \subset \mathbb{R}^2,\\
    \label{eq: wall boundary condition} v \cdot n & = 0 & & \text{on } \partial \Omega,\\
    u & = u_0 & & \text{on } \Omega \text{ at } t=0,%\Omega \times \{ t = 0 \},
\end{alignat}
\end{subequations}
where $n$ is the outward unit normal vector and $u_0$ the initial condition. Equation~\eqref{eq: wall boundary condition} specifies a reflecting wall boundary condition.

\subsection{Base discretization}
To discretize the domain $\Omega$, we first start with a rectangular domain $\widehat{\Omega} \supset \Omega$ and a so-called background mesh $\widehat{\mathcal{M}}_h$ of $\widehat{\Omega}$, a structured Cartesian mesh
of mesh size $h$,
such that $\bigcup_{\widehat{E} \in \widehat{\mathcal{M}}_h} \overline{\widehat{E}} = \overline{\widehat{\Omega}}$.

With this we define a cut-cell mesh on $\Omega$ consisting of elements
\[
\mathcal{M}_h = \{ E = \widehat{E} \cap \Omega : \widehat{E} \in \widehat{\mathcal{M}}_h \},
\]
and internal and boundary faces given by
\begin{align*}
    \GammaInt & = \{ \gamma = \bar{E}_1 \cap \bar{E}_2 : |\gamma| > 0, E_1, E_2 \in \mathcal{M}_h \},\\
    \GammaExt & = \{ \gamma = \bar{E} \cap \partial \Omega : E \in \mathcal{M}_h \},\\
    \Gamma_h & = \GammaInt \cup \GammaExt,
\end{align*}
where $|\gamma|$ denotes the 1D volume (length).
Given an element $E \in \Mesh$ we denote its sets of
internal, boundary, and all faces by
\begin{gather*}
\GammaInt(E) = \{ \gamma \in \GammaInt : \gamma \cap \partial E \neq \emptyset \}, \quad \GammaExt(E) = \{ \gamma \in \GammaExt : \gamma \cap \partial E \neq \emptyset \}, \\%\quad
\text{and} \quad \Gamma_h(E) = \GammaInt(E) \cup \GammaExt(E).
\end{gather*}
We further define for each element $E$ the index-set $\mathbb{I}(E) = \{1, \ldots, |\Gamma_h(E)| \}$, consecutively numbering all faces, such that any $i \in \mathbb{I}(E)$ identifies exactly one face $\gamma_i \in \Gamma_h(E)$. If $\gamma_i \in \GammaInt(E)$ is an internal face we denote by $E_i$ the corresponding neighbor element. If $\gamma_i \in \GammaExt(E)$ and there is no physical neighbor at this face, we will abuse notation by still writing $E_i$ and using this as an index variable later. Formally we can think of $E_i$ as a ghost cell which shares the face $\gamma_i$ with $E$.

For each internal face $\gamma = \bar{E}_1 \cap \bar{E}_2 \in \GammaInt$ we fix a unique left-sided cell $E_{\gamma}^L \in \{E_1, E_2 \}$. The corresponding right-sided cell will be denoted by $E_{\gamma}^R$. 
The unit normal vector $n_{\gamma}$ on $\gamma$ is uniquely defined to point from left to right, i.e.,
\[
n_{\gamma} = n_{E_{\gamma}^L} |_{\gamma}.
\]
Given a face $\gamma = \bar{E}_{\gamma}^L \cap \bar{E}_{\gamma}^R \in \GammaInt$ and a function $w$ defined on $E_{\gamma}^L \cup E_{\gamma}^R$ such that both $w|_{E_{\gamma}^L}$ and $w|_{E_{\gamma}^R}$ have well-defined traces on $\gamma$, we define left and right values of $w$ as
\[
w_{\gamma}^L = \trace{w|_{E_{\gamma}^L}}, \; w_{\gamma}^R = \trace{w|_{E_{\gamma}^R}},
\]
and averages and jumps
\[
\avg{w}_{\gamma} = \frac{1}{2} (w^L_{\gamma} + w^R_{\gamma}), \quad \llbracket w \rrbracket_{\gamma} = w^L_{\gamma} - w^R_{\gamma}.
\]

In the DG discretization, the reflecting boundary condition $v \cdot n = 0$ is not imposed strongly; instead, it is implemented in a weak sense by constructing suitable numerical fluxes on the boundary. To do so we define the mirroring operator.
\begin{definition}[Mirroring operator]
We introduce a mirroring operator that represents a suitable notion of a reflected state on the boundary, see also \cite{dgm_dol_fei}. In the course of the paper we will further need to mirror points, as well as states. By abuse of notation we will denote all these operations by $\mathfrak M_n$.

Let a cell $E \in \Mesh$ and a boundary face $\gamma \in \GammaExt(E)$ be given.
For a state $u=(p,v)^T$ we define the state mirrored at $\gamma$ component wise as
\begin{equation} \label{eq: reflecting wall boundary operator}
\mathfrak{M}_n : \mathbb{R}^3 \to \mathbb{R}^3,
\quad
\begin{pmatrix}
    p \\ v
\end{pmatrix} \mapsto 
\begin{pmatrix}
    p \\ v - 2 (v \cdot n) n
\end{pmatrix}.
\end{equation}

This definition is consistent, as for the exact solution with $v \cdot n = 0$ the reflected velocity simply changes the sign of the normal component. The pressure should be continuous under reflection, which is also reflected in the definition. If an approximate $\tilde v$ does not fulfill the boundary condition the operator introduces a jump such that the average $\tfrac 1 2 (\tilde v + \mathfrak M_n(\tilde v)) = 0$ and the jump will later show up as a penalty, weakly enforcing the boundary condition.

We further introduce an extension of this definition to the whole element $E$, allowing to mirror a complete vector-valued polynomial function $u \in \mathcal{P}^r(E)^3$ instead of just a single pointwise state. We first denote the orthogonal projection of $x \in E$ onto $\gamma$ as $x^\perp$ and $n_\gamma^\perp$ as the unit outer normal in $x^\perp$.
With this we define the generalized mirroring operator
\begin{equation}\label{eq: generalized reflecting wall boundary operator}
    \mathfrak{M}_n : \mathcal{P}^r(E)^3 \to \mathcal{P}^r(E)^3,
\quad
\bigg( x \mapsto
\begin{pmatrix}
    p(x) \\ v(x)
\end{pmatrix} \bigg) \mapsto 
\bigg(x \mapsto
\begin{pmatrix}
    p(x) \\ v(x) - 2 (v(x^\perp) \cdot n) n
\end{pmatrix}\bigg),
\end{equation}
which allows to evaluate $\mathfrak M_n(u)(x) ~\forall x\in E$. We note that this definition coincides on $\gamma$ with the definition of mirroring the state \eqref{eq: reflecting wall boundary operator}.
\end{definition}

\begin{remark}[Skew symmetry]
Note that $\mathfrak M_n$ is skew-symmetric on $\gamma$:
\begin{equation}\label{eq: boundary reflection skew symmetry}
\langle A_n \mathfrak{M}_n(u), w \rangle = - \langle A_n u, \mathfrak{M}_n(w) \rangle.
\end{equation}
This is an important property of as it leads to energy conservation on reflecting boundaries.
\end{remark}

We can now state the standard DG discretization in  space that will later be stabilized appropriately. As a function space we choose
\[
\dspace \coloneqq \{ v_h \in (L^2(\Omega))^3 \; \big| \; \forall E \in \mathcal{M}_h : (v_h)_E \coloneqq (v_h)_{|E} \in P^r(E)^3 \},
\]
with $P^r(E)$ being the space of polynomials of degree $r$ on $E$.
On this space we consider discretizations that can be formulated as a bilinear operator $a_h:\dspace \times \dspace \to \mathbb{R}$ using central fluxes and an additional diffusive stabilization operator $s_h:\dspace \times \dspace \to \mathbb{R}$:
\begin{align}
\label{eq:base-bilinear-form}
   \begin{split}
    a_h(u_h, w_h) = & - \sum_{E \in \mathcal{M}_h} \int_E f(u_h) \cdot \nabla w_h + \sum_{\gamma \in \GammaInt} \int_{\gamma} \langle \tfrac{1}{2}\big[f_n(u_{\gamma}^L) + f_n(u_{\gamma}^R)\big], \llbracket w_h \rrbracket \rangle\\
    & + \sum_{\gamma \in \GammaExt} \int_{\gamma} \langle \tfrac{1}{2} \big[f_n(u_h) +  f_n(\mathfrak{M}_n(u_h)) \big] , w_h \rangle
\end{split}
\end{align}
where $\langle \cdot, \cdot \rangle$ denotes the Euclidean scalar product on $\mathbb{R}^3$ and
\begin{align} \label{eq: base scheme dissipative term}
   \begin{split}
    s_h(u_h, w_h) = & \phantom{\: +} \sum_{\gamma \in \GammaInt} \int_{\gamma} \langle S_n(u^L_{\gamma}, u^R_{\gamma}), \llbracket w_h \rrbracket \rangle\\
    & + \sum_{\gamma \in \GammaExt} \int_{\gamma} \langle S_n(u_h, \mathfrak{M}_n(u_h)) , w_h \rangle
\end{split}
\end{align}
where
\[
S_n: \mathbb{R}^3 \times \mathbb{R}^3 \to \mathbb{R}^3
\]
with unit vector $n$ is a dissipative operator. Here $u^L_{\gamma}$ and $u^R_{\gamma}$ denote the state in the left or right neighbor, as defined above. 
Note that throughout the paper we will often consider two arbitrary states, not directly connected to a specific face. In this case we use the general notation of $u^\mu$ and $u^\nu$ for these two states, to indicate that they do not belong to a mesh face.

We assume the diffusive operator $S_n$ to be linear, i.e. $S_n(u^\mu, u^\nu) = S_n^M (u^\mu - u^\nu)$ with $S_n^M$ being positive semi-definite. A common choice for the wave-equation is the Lax-Friedrichs dissipation 
\begin{equation} \label{eq: lax-friedrichs dissipation}
S_n(u^\mu, u^\nu) = \frac{\rho(A_n)}{2} (u^\mu - u^\nu)
\end{equation}
with $\rho(A_n)$ being the spectral radius of the matrix $A_n$ defined in \eqref{eq: directional flux matrix}, which is equal to the speed of sound $c$ as introduced in \eqref{eq:physical-flux}.

The complete base semi-discrete scheme reads as follows: Find $u_h(t) \in \dspace$ such that
\begin{equation}\label{eq: base scheme}
     \int_{L^2(\Omega)} \langle \partial_t u_h(t), w_h \rangle  + a_h(u_h(t), w_h) + s_h(u_h(t), w_h) = 0 \quad \forall w_h \in \dspace.
\end{equation}
The scheme is then completed by applying an explicit Runge-Kutta (RK) scheme in time.

\subsection{Surface forms and numerical fluxes}
We can formulate the scheme in terms of numerical fluxes $\mathcal{H}_n$. Setting 
\[
\mathcal{H}_n(u^\mu, u^\nu) = \frac{1}{2} \big [f_n(u^\mu) + f_n(u^\nu)\big] + S_n(u^\mu, u^\nu)
\]
the discretization \eqref{eq: base scheme} can be written as
\[
     \int_{L^2(\Omega)} \langle \partial_t u_h(t), w_h \rangle  + a^{\text{num}}_h(u_h(t), w_h) = 0 \quad \forall w_h \in \dspace
\]
with
\begin{align*}
    a^{\text{num}}_h(u_h, w_h) & = - \sum_{E \in \mathcal{M}_h} \int_E f(u_h) \cdot \nabla w_h + \sum_{\gamma \in \GammaInt} \int_{\gamma} \langle \mathcal{H}_n(u^L_{\gamma}, u^R_{\gamma}), \llbracket w_h \rrbracket \rangle\\
    & \phantomminus + \sum_{\gamma \in \GammaExt} \int_{\gamma} \langle \mathcal{H}_n(u_h, \mathfrak{M}_n(u_h)), w_h \rangle\\
    & =  a_h(u_h, w_h) + s_h(u_h, w_h).
\end{align*}
Note that with the choice \eqref{eq: lax-friedrichs dissipation}, this simply corresponds to using the Lax-Friedrichs flux in a standard DG formulation.

Although this is the common form when working with hyperbolic conservation laws, splitting the numerical flux into central and dissipative parts will significantly simplify our presentation due to their different nature. To see this and for later purposes, we introduce the notation of surface forms $b_{\gamma}: L^2(\gamma)^3 \times L^2(\gamma)^3 \times L^2(\gamma)^3 \to \mathbb{R}$
\begin{equation} \label{eq: surface forms}
b_{\gamma}(u^\mu, u^\nu, w) = \int_{\gamma} \langle \tfrac{1}{2}\big [f_n(u^\mu) + f_n(u^\nu) \big], w \rangle,
\end{equation}
where $\gamma \in \Gamma_h$ is a mesh face.
On the one hand we have the identity
\begin{align*}
    a_h(u_h, w_h) = & - \sum_{E \in \mathcal{M}_h} \left( \int_E f(u_h) \cdot \nabla w_h + \sum_{\gamma \in \GammaInt(E)} b_{\gamma} (u_{\gamma}^L, u_{\gamma}^R, w_h|_E) \right. \\
   & \left. + \sum_{\gamma \in \GammaExt(E)} b_{\gamma} (u_h, \mathfrak{M}_n(u_h) , w_h) \right)
\end{align*}
which connects these surface forms to the space discretization. On the other hand we have the following statement.
\begin{lemma}
Let $E \in \Mesh$ be a mesh element, $u^\mu, u^\nu, w \in \mathcal{P}^r(E)^3$ vector-valued polynomial functions. Then it holds  
\begin{equation} \label{eq: surface forms int by parts}
\sum_{\gamma\in \Gamma_h(E)} b_{\gamma}^E(u^\mu, u^\nu, w) = \int_E \tfrac{1}{2} [f(u^\mu) + f(u^\nu)] \cdot \nabla w  + \int_E   \tfrac{1}{2} \langle \nabla \cdot [f(u^\mu) + f(u^\nu)], w \rangle.
\end{equation}
\end{lemma}
\begin{proof}
By integration by parts, we have
\begin{align*}
\sum_{\gamma\in \Gamma_h(E)} \int_{\gamma} \langle \tfrac{1}{2} f_n(u^\mu), w \rangle & = \int_E \sum_{k=1}^d \langle \tfrac{1}{2} f_k(u^\mu), \partial_k w \rangle + \int_E \sum_{k=1}^d \langle \tfrac{1}{2} \partial_k f_k(u^\mu), w \rangle\\
& = \int_E \tfrac{1}{2} f(u^\mu) \cdot \nabla w  + \int_E   \tfrac{1}{2} \langle \nabla \cdot f(u^\mu), w \rangle
\end{align*}
and similarly for $u^\nu$. Adding both equations together yields the result.
\end{proof}
So for the central flux, we have an integration by parts formula that can also be seen as a balance of fluxes. It is this balance of fluxes that we will exploit to construct our method. Constructing a similar balance for dissipative fluxes is highly non-trivial and does not seem to provide any benefit. Splitting our numerical fluxes in central and dissipative parts thus seems natural.

\section{Domain of Dependence Stabilization} \label{sec: dod section}
In order to overcome the small cell problem and allow explicit time stepping methods with a CFL condition that is independent of the cut-cell size, we now introduce a generalization of the original DoD discretization, as it was described in \cite{DoD_2d_linadv_2020}.

As in \cite{DoD_2d_linadv_2020} we add stabilization terms on small cut cells to overcome these time step restrictions. Whether a cell is small is decided based on its volume size relative to its background cell.

\begin{definition}[Cell classification]
We denote the volume fraction of an element $E \in \Mesh$ by
\[
\alpha_E = \frac{|E|}{|\widehat{E}|},
\]
where $\widehat{E}$ is the corresponding background cell. We then define {\em the set of stabilized cut cells} as
\begin{equation}\label{eq: set of stab cut cells}
\cI = \cI_{\alpha} = \{ E \in \mathcal{M}_h \; | \; \alpha_E < \alpha \}
\end{equation}
with $\alpha \in (0, 1)$ being the threshold for volume fractions of cut cells that will be stabilized.
\end{definition}
Typically not all cut cells need stabilization and often it is advisable to lower the time step size slightly to allow for a slightly smaller $\alpha$. In section~\ref{sec: numerical section} we will give a concrete choice for our chosen test cases.

In the following we will make a few assumptions on the set $\cI$ of small cut cells.
\begin{enumerate}[label=(\roman*)]
    \item \label{item: no small neighbors} For any small cut-cell $E \in \cI$, there is no $i \in \IE$ such that $E_i \in \cI$ or in other words, small cut cells do not have small neighbors.
    \item \label{item: reflecting wall face assumption} For any small cut-cell $E \in \cI$, if there is a face $\gamma \in \GammaExt(E)$, its outer normal vector $n_\gamma$ is constant. In other words, $\gamma$ is a subset of a hyperplane in $\mathbb{R}^2$.
    \item \label{item: no anisotropic small cells}There is a constant $\rho > 0$ such that for any $E \in \cI$ it holds that $h_E^d \leq \rho |E|$ where $h_E$ is the diameter of the cell $E$.
\end{enumerate}

\begin{remark}
Assumption \ref{item: no small neighbors} is the same as in \cite{DoD_2d_linadv_2020}, where as condition~\ref{item: no anisotropic small cells}
limits the amount of anisotropy in the cut-cell mesh (at least for small cells).
The later is only currently needed due to technical limitations of our implementation. Assumption \ref{item: reflecting wall face assumption} is a restriction on the geometry, which relates to the handling of reflecting boundary conditions.

The general framework that we present here does
in principal
not require the conditions \ref{item: reflecting wall face assumption} and \ref{item: no anisotropic small cells}. It is our concrete construction and implementation that we used for our numerical results in section \ref{sec: numerical section} which is limited in this regard. Condition \ref{item: no small neighbors} is more fundamental and will require an extension of our general framework that can handle patches of small cut cells.
\end{remark}

Given these assumptions, we will now construct the DoD stabilized semi-discrete scheme. It takes the form: Find $u_h(t) \in \dspace$ such that for all $w_h \in \dspace$
\begin{equation}
\label{eq: stabilized scheme}
\int_{\Omega} \langle d_t u_h(t), w_h \rangle  + a_h(u_h(t), w_h) + J^a_h(u_h(t), w_h) + s_h(u_h(t), w_h) + J^s_h(u_h(t), w_h) = 0.
\end{equation}
The stabilization terms $J^a_h$ and $J^s_h$ contain element stabilization terms for all cells $E \in \mathcal{I}$ that are considered small and are given as
\begin{subequations}
\label{eq: cell wise stabilization}
\begin{align}
J^{a}_h(u_h, w_h) & = \sum_{E \in \mathcal{I}} J^{0, E}_h(u_h, w_h) + \sum_{E \in \mathcal{I}} J^{1, E}_h(u_h, w_h) \\
J^{s}_h(u_h, w_h)  & = \sum_{E \in \mathcal{I}} J^{s, E}_h(u_h, w_h).
\end{align}
\end{subequations}
Before giving the definitions of $J_h^{0,E}$, $J_h^{1,E}$ and $J^{s, E}_h$ we introduce some more terminology and concepts. First let us recall the \textit{small cell problem}: Using a time step size $\Delta t = \mathcal{O}(h)$, with $h$ being the mesh size of the background mesh, will lead to over-/undershoots for updates on the cut-cell mesh. This behavior is related to characteristics reaching beyond the immediate neighborhood of an element in the mesh, as illustrated in figure~\ref{fig: small cell characteristics}. This prevents information from reaching its proper destination. We follow the main idea behind the DoD stabilization, which was already laid out in \cite{DoD_2d_linadv_2020}, and construct the stabilization terms in such a way that the numerical domain of dependence of elements in the neighborhood of a small cell is extended, see figure~\ref{fig: domain of dependence} for an illustration. To do so, we will work with extended supports of local functions.

\begin{figure}[thp]
    \centering
    \begin{tikzpicture}
    \begin{scope}[xscale=2.0]
    \draw (-2.5, 0.0) -- (2.5, 0.0);
    \draw (0.0, -0.1) -- (0.0, 0.1);
    \draw (0.3, -0.1) -- (0.3, 0.1);
    \draw (1.0, -0.1) -- (1.0, 0.1);
    \draw (2.0, -0.1) -- (2.0, 0.1);
    \draw (-1.0, -0.1) -- (-1.0, 0.1);
    \draw (-2.0, -0.1) -- (-2.0, 0.1);

    \draw (-2.5, 1.0) -- (2.5, 1.0);
    \draw (0.0, 0.9) -- (0.0, 1.1);
    \draw (0.3, 0.9) -- (0.3, 1.1);
    \draw (1.0, 0.9) -- (1.0, 1.1);
    \draw (2.0, 0.9) -- (2.0, 1.1);
    \draw (-1.0, 0.9) -- (-1.0, 1.1);
    \draw (-2.0, 0.9) -- (-2.0, 1.1);

    \node at (0.15, -0.25) {$E$};
    \node at (-0.45, -0.25) {$E_{\text{up}}$};
    \node at (0.7, -0.25) {$E_{\text{dw}}$};

    \draw[orange, thick] (1.0, 1.0) -- (0.5, 0.0);
    \draw[orange, thick] (0.3, 1.0) -- (-0.2, 0.0);
    \draw[orange, thick] (0.4, 1.0) -- (-0.1, 0.0);
    \draw[orange, thick] (0.5, 1.0) -- (0.0, 0.0);
    \draw[orange, thick] (0.6, 1.0) -- (0.1, 0.0);
    \draw[orange, thick] (0.7, 1.0) -- (0.2, 0.0);
    \draw[orange, thick] (0.8, 1.0) -- (0.3, 0.0);
    \draw[orange, thick] (0.9, 1.0) -- (0.4, 0.0);

    \node[anchor=west] at (-3.0, 0.0) {$t^n$};
    \node[anchor=west] at (-3.0, 1.0) {$t^{n+1}$};
    \end{scope}
    \end{tikzpicture}
    \caption{Characteristics entering $E_{\text{dw}}$ in the time interval $[t^n, t^{n+1}]$. Since the neighbor cell $E$ is too small, these characteristics even reach into $E_{\text{up}}$ at time $t^n$.}
    \label{fig: small cell characteristics}
\end{figure}
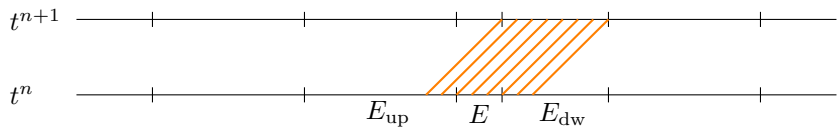

\begin{definition}[Extension operator] We follow the definition proposed in \cite{DoD_2d_linadv_2020}.
Given a discrete function $w_h \in \dspace$ and a cell $E \in \mathcal{M}_h$ we introduce the extension operator
\[
\extop{E}: \dspace \to (P^r(\Omega))^3, \quad w_h \mapsto \extop{E}(w_h)
\]
defined via
\[
\extop{E}(w_h)(x) = w_h|_{E}(x), \quad x \in E,
\]
i.e., we select the polynomial components of $w_h$ on $E$ and extend them to the complete domain. This definition is well-posed since defining a polynomial function on an open domain is equivalent to defining it on the complete real space.
\end{definition}

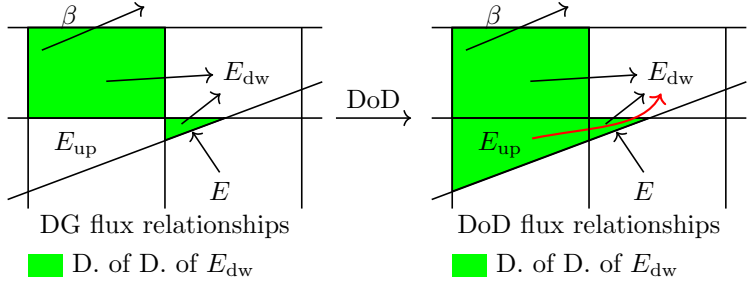
\begin{figure}
    \centering
    	\begin{tikzpicture}[color=black,semithick,yscale=0.6, xscale=0.9]
		
		\begin{scope}[shift={(-2.6, 0.0)}]
			% small cell
			\draw[fill=green] (0.0, -0.49) -- (0.0, 0.0) -- (0.875, 0.0) --cycle;

			% E_up
			\draw[fill=green](-2,0) -- (-2, -1.62) -- (0, -0.49)--  (0.0, 0.0) --cycle;
			
			% bigger neighbor
			\draw[fill=green](-2,0) -- (-2,  2) -- (0, 2)--  (0.0, 0.0) --cycle;
			
			% legend
			\fill[fill=green](-2,-3) -- (-2,  -3.5) -- (-1.5, -3.5)--  (-1.5,  -3) --cycle;
			\node at (-1.5, -3.25) [right]{D. of D. of $E_{\text{dw}}$};
			% horizontal lines
			\draw(-2.3, 2) -- (2.3, 2);
			\draw(-2.3, 0) -- (2.3, 0);
			
			% cut line
			\draw(-2.3, -1.8) -- (2.3, 0.8);
			
			\draw(-0,0) -- (-2.0, 0) -- (-2.0, 2.0) -- (0.0, 2.0) -- cycle;
			
			\draw (1.2, 1) node (2){$E_{\text{dw}}$};
			\draw (-1.3, -0.6) node (1){$E_{\text{up}}$};
			\draw[->] (0.8, -1.2) node[anchor=north, color=black] {$E$} -- (0.4, -0.3);
			
			\node at (-1.0, 0.8) (3) {};
			\node at (0.1, -0.3) (4) {};
			
			\draw[->, red, thick] (1) to [out=15, in=250] (2);
			\draw[->] (3) to (2);
			\draw[->] (4) to (2);
			
			% beta
			\draw[->] (-1.8, 1.5) -- (-1.4, 1.75) node[anchor=south]{$\beta$}  -- (-0.3, 2.46);

			% vertical lines
			\draw(-2, 2) -- (-2, -2);
			\draw(0, 2) -- (0, -2);
			\draw(2, 2) -- (2, -2);
			
			\node at (-0.0, -2.4) {DoD flux relationships};
			
		\end{scope}
		
		\draw[->] (-6.3, 0.0) -- (-5.75, 0.0) node[anchor=south]{DoD}  --  (-5.2, 0.0);

		\begin{scope}[shift={(-8.8, 0.0)}]
			% horizontal lines
			\draw[fill=green] (0.0, -0.49) -- (0.0, 0.0) -- (0.875, 0.0) --cycle;
			\draw(-2.3, 2) -- (2.3, 2);
			\draw(-2.3, 0) -- (2.3, 0);
			
			\draw (1.2, 1) node (6) {$E_{\text{dw}}$};
			\draw (-1.3, -0.6) node (5) {$E_{\text{up}}$};
			\draw[->] (0.8, -1.2) node[anchor=north, color=black] {$E$} -- (0.4, -0.3);
			
						% bigger neighbor
			\draw[fill=green](-2,0) -- (-2,  2) -- (0, 2)--  (0.0, 0.0) --cycle;
			
						% legend
			\fill[fill=green](-2,-3) -- (-2,  -3.5) -- (-1.5, -3.5)--  (-1.5,  -3) --cycle;
			\node at (-1.5, -3.25) [right]{D. of D. of $E_{\text{dw}}$};
			
			\node at (-1.0, 0.8) (7) {};
			\node at (0.1, -0.3) (8) {};
			
			\draw[->] (7) to (6);
			\draw[->] (8) to (6);
			
			% cut line
			\draw(-2.3, -1.8) -- (2.3, 0.8);
			
			\draw(-0,0) -- (-2.0, 0) -- (-2.0, 2.0) -- (0.0, 2.0) -- cycle;
			
			% beta
			\draw[->] (-1.8, 1.5) -- (-1.4, 1.75) node[anchor=south]{$\beta$}  -- (-0.3, 2.46);
			
			% vertical lines
			\draw(-2, 2) -- (-2, -2);
			\draw(0, 2) -- (0, -2);
			\draw(2, 2) -- (2, -2);
			
			\node at (-0.0, -2.4) {DG flux relationships};
		\end{scope}
		
	\end{tikzpicture}
    \caption{Extension of the domain of dependence of a small cell $E$ via our stabilization method displayed for the linear transport equation with a transport vector $\beta$ which is parallel to the cut face. On the left we see the domain of dependence of the cell $E_{\text{dw}}$ when the base DG scheme~\eqref{eq: base scheme} is applied. On the right we see the domain of the dependence of the same element in the mesh when our stabilized scheme~\eqref{eq: stabilized scheme} is applied. Now the cell $E_{\text{dw}}$ will receive information from $E_{\text{up}}$ in a single time step/stage even when a large time step size is used.}
    \label{fig: domain of dependence}
\end{figure}

In order to stabilize reflecting boundaries in small cut cells we need to generalize the above definition to reflecting faces.
\begin{definition}[Reflected extension operator]
Let $w_h \in \dspace$ be a discrete function, $E \in \mathcal{M}_h$ an element and $\gamma \in \GammaExt$ a boundary face on which a reflecting wall boundary condition is applied.
We define the extension operator over this face $\gamma$ as
\begin{equation}
\label{eq:mirroring-extension-op}
\mathcal{L}_{E}^{\gamma}: \dspace \to P^r(\Omega)^3,
\quad w_h \mapsto \mathcal{L}_{E}^{\gamma}(w_h) = \mathfrak{M}_{n_{\gamma}}(\extop{E}(w_h)).
\end{equation}
\end{definition}
To unify the application of the different extension operators, we introduce additional notation. Let $E \in \cI$ be a small cut cell with $\gamma_i \in \Gamma_h(E)$ and $\gamma_j \in \Gamma_h(E)$ being two of its faces, $i, j \in \IE$. Then we denote by
\begin{equation} \label{eq: unified extension operator}
\extop{\mathcal{E}}^{ij} =
\begin{cases}
\extop{\mathcal{E}}
&\quad\text{if }
\mathcal{E} = E \text{ or } \mathcal{E} = E_k, k \in \{i, j\} \text{ and } \gamma_k \in \GammaInt\\
\extop{E_j}^{\gamma_i}
&\quad\text{if } \mathcal{E} = E_i \text{ and }
\gamma_i \not \in \GammaInt\\
\extop{E_i}^{\gamma_j}
&\quad\text{if } \mathcal{E} = E_j \text{ and }
 \gamma_j \not \in \GammaInt
\end{cases}
\end{equation}
the extension operator between $\gamma_i$ and $\gamma_j$ that selects either of the previously defined extension operators based on whether we are extending from an internal or external boundary. Note that we abuse notation here by writing by writing $\mathcal{E} = E_i$ even if $\gamma_i \not \in \GammaInt$ for $i \in \IE$. Note also that we have suppressed the dependence on $E$ here for brevity.

In figure~\ref{fig: domain of dependence} we displayed an example for the linear transport equation with a flux direction parallel to the cut interface. In this case only interactions between two different faces have to be considered, due to a no-flow condition on the cut cell, see also \cite{DoD_2d_linadv_2020}. But here we are now considering the linear wave equation. The behavior of its characteristics is very different and resembles a cone-like structure in space-time. Then interactions between all faces with their respective neighbor elements have to be taken into account to properly extend the domain of dependence. To do so, we introduce the following concept:

\begin{definition}[Propagation forms] \label{def: propagation forms}
Let $E \in \cI$ be a small cell. We introduce a set of surface and volume propagation forms as functionals
\begin{align*}
  p_{ij}^E(u^{\mu}, u^{\nu}, w) & : (P^r(E))^3 \times (P^r(E))^3 \times (P^r(E))^3 \to \mathbb{R}
\intertext{and}
  p_V^E(u^{\mu}, u^{\nu}, w) & : (P^r(E))^3 \times (P^r(E))^3 \times (P^r(E))^3 \to \mathbb{R}\\
  p_V^{E, *}(u^{\mu}, u^{\nu}, w) & : (P^r(E))^3 \times (P^r(E))^3 \times (P^r(E))^3 \to \mathbb{R},
\end{align*}
where $p_V^{E, *}$ should mimic a kind of an adjoint operator to $p_V^{E}$. These forms are required to be symmetric in the first two arguments and linear in the third. Furthermore they have to satisfy for all vector-valued polynomial functions $u^{\mu}, u^{\nu}, w \in \mathcal{P}^r(E)^3$ the following properties:
\begin{enumerate}[i),labelindent=0pt,listparindent=0pt,align=left,leftmargin=0pt,labelwidth=0pt,itemindent=!]
\begin{subequations}
    \item Let $i, j \in \mathbb{I}(E)$. Then it holds
    \begin{equation} \label{eq: propagation balance}
        p_{ij}^E(u^{\mu}, u^{\nu}, w) + p_{ji}^E(u^{\mu}, u^{\nu}, w) = p_V^E(u^{\mu}, u^{\nu}, w) + p_V^{E, *}(u^{\mu}, u^{\nu}, w).
    \end{equation}
    \item \label{lem: propagation forms consistency} Let $j \in \mathbb{I}(E)$. Then it holds
    \begin{equation} \label{eq: propagation forms consistency}
        \sum_{i \in \mathbb{I}(E), i \neq j} p_{ij}^E(u^{\mu}, u^{\nu}, w) = b_{\gamma_j}(u^{\mu}, u^{\nu}, w)
    \end{equation}
\end{subequations}
with $b_{\gamma_j}$ from \eqref{eq: surface forms}.
\end{enumerate}
\end{definition}

\begin{remark}
Equation~\eqref{eq: propagation balance} can be seen as a variation of the integration by parts property, compare with~\eqref{eq: surface forms int by parts} which has a similar structure, i.e. a balance between surface and volume integral terms. Equation~\eqref{eq: propagation forms consistency} is essentially a consistency property. The idea is that all fluxes coming from neighbor elements and going into a particular cell will add up to the standard DG flux from the base scheme. Similar ideas appeared already in \cite{fvca_proc_dod} but in a more concrete shape based on products of flux matrices. The additional freedom that we gain from working with abstract forms is a key step in constructing a high-order method at small boundary cells.
\end{remark}

The propagation forms are the key concept that enables us to extend the DoD method to the wave equation and to prove energy conservation, however we have only given an abstract definition. The properties~\eqref{eq: propagation balance} and~\eqref{eq: propagation forms consistency} define a linear system which is underdetermined. A particular example extending the concept of a central flux is the following:
\begin{example}[Central propagation forms] \label{exam: central propagation forms}
Let $E \in  \cI$ and $i, j \in \mathbb{I}(E)$. Let $K = |\mathbb{I}(E)|$ be the number of faces of $E$. The central (surface) propagation form on $E$ from $E_i$ to $E_j$ is given by
  \begin{equation} \label{eq: central surface propagation form}
    p^{E, Z}_{ij}(u^{\mu}, u^{\nu}, w) = \tfrac{1}{K-1} b_{\gamma_j}(u^{\mu}, u^{\nu}, w) - \tfrac{K-2}{K(K-1)}  b_{\gamma_i}(u^{\mu}, u^{\nu}, w) + \tfrac{1}{K(K-1)} \sum_{\substack{k \in \mathbb{I}(E)\\k \neq i,j}} b_{\gamma_k}(u^{\mu}, u^{\nu}, w).
\end{equation}
The central volume propagation forms on $E$ are given by
\begin{equation} \label{eq: central volume propagation form}
\begin{split}
       p_V^{E, Z}(u^{\mu}, u^{\nu}, w) & = \tfrac{2}{K(K-1)} \int_E  \tfrac{1}{2}\big[f_k(u^{\mu}) + f_k(u^{\nu}) \big] \cdot \nabla w, \\ p_V^{E, Z, *}(w^{\mu}, w^{\nu}, u) & = \tfrac{2}{K(K-1)} \int_E \nabla \cdot \tfrac{1}{2} \big[f(w^{\mu}) + f(w^{\nu}) \big] u.
\end{split}
\end{equation}
These constitute a set of surface and volume propagation forms on $E$, i.e. they fulfill the properties~\eqref{eq: propagation balance} and~\eqref{eq: propagation forms consistency}:
    \begin{enumerate}[i),labelindent=0pt,listparindent=0pt,align=left,leftmargin=0pt,labelwidth=0pt,itemindent=!]
        \item We have
        \begin{align*}
        p_{ij}^{E, Z}(u^{\mu}, u^{\nu}, w) + p_{ji}^{E, Z}(u^{\mu}, u^{\nu}, w) & = \tfrac{2}{K(K-1)} \biggl( b_{\gamma_i}(u^{\mu}, u^{\nu}, w) +  b_{\gamma_j}(u^{\mu}, u^{\nu}, w) + \sum_{\substack{k \in \mathbb{I}(E)\\k \neq i,j}} b_{\gamma_k}(u^{\mu}, u^{\nu}, w) \biggr)\\
        & = p_{V}^{E, Z}(u^{\mu}, u^{\nu}, w) + p_{V}^{E, Z, *}(u^{\mu}, u^{\nu}, w)
        \end{align*}
        where the last equation follows from~\eqref{eq: surface forms int by parts} and~\eqref{eq: central volume propagation form}. This establishes~\eqref{eq: propagation balance}.
        \item By definition and using $K = \IE$ we get
        \begin{align*}
        \sum_{i \in \mathbb{I}(E), i \neq j} p_{ij}^{E, Z}(u^{\mu}, u^{\nu}, w) & \overset{\mathclap{\eqref{eq: central surface propagation form}}}{=} \phantomplus \sum_{i \in \mathbb{I}(E), i \neq j} \tfrac{1}{K-1} b_{\gamma_j}(u^{\mu}, u^{\nu}, w) - \sum_{i \in \mathbb{I}(E), i \neq j} \tfrac{K-2}{K(K-1)} b_{\gamma_i}(u^\mu, u^\nu, w)\\
        & \phantomeq \binplus \sum_{i \in \mathbb{I}(E), i \neq j} \tfrac{1}{K(K-1)} \sum_{k \in \mathbb{I}(E),k \neq i,j} b_{\gamma_k}(u^\mu, u^\nu, w)\\
        & = \phantomplus b_{\gamma_j}(u^{\mu}, u^{\nu}, w) - \sum_{i \in \mathbb{I}(E), i \neq j} \tfrac{K-2}{K(K-1)} b_{\gamma_i}(u^\mu, u^\nu, w)\\
        & \phantomeq \binplus \sum_{i \in \mathbb{I}(E), i \neq j} \tfrac{1}{K(K-1)} \sum_{k \in \mathbb{I}(E),k \neq i,j} b_{\gamma_k}(u^{\mu}, u^{\nu}, w)\\
        & = b_{\gamma_j}(u^\mu, u^\nu, w). 
        \end{align*}
        For the last step note that any $b_{\gamma_i}(u^\mu, u^\nu, w)$, $i \in \IE, i \neq j$, is exactly $K-2$ times contained  in the sum
        \[
        \sum_{i \in \mathbb{I}(E), i \neq j} \tfrac{1}{K(K-1)} \sum_{k \in \mathbb{I}(E),k \neq i,j} b_{\gamma_k}(u^{\mu}, u^{\nu}, w).
        \]
        Thus \eqref{eq: propagation forms consistency} does hold as well.
    \end{enumerate}
\end{example}

The propagation forms allow us to define in what direction information flows but they do not specify the precise amount. For this we introduce the capacity of a small cell and its stabilization parameter:

\begin{definition}[Cell capacity]
Let $E \in \cI$ be a small cell and $r$ be the polynomial degree used in the definition of $\dspace$. The capacity of $E$ is defined as
\begin{equation}
\capE(\Delta t) = \frac{1}{2 r + 1}\frac{|E|}{ \Delta t c \max_{i \in \mathbb I(E)} |\gamma_i|}
\end{equation}
and its stabilization parameter as $\eta_E = \eta_E(\Delta t) = 1 - \capE$.
\end{definition}
Now we are ready to give the definitions of $J_h^{0,E}$, $J_h^{1,E}$ and $J^{s, E}_h$ in~\eqref{eq: cell wise stabilization} to complete the definition of our stabilized scheme~\eqref{eq: stabilized scheme}. The stabilization of a small cut-cell $E \in \cI$ is based on a direct coupling of neighboring cells, effectively tunneling information through the small cut cell.
\begin{definition}[Cell stabilization terms] {\label{def: cell stabilization terms}}
Let $E \in \cI$ be a small cut cell. The cell-wise stabilization terms $J_{h}^{0,E}$, $J_{h}^{1,E}$ and $J_{h}^{s,E}$ are constructed from pairwise (for each pair of neighbors $i, j \in \IE, i \neq j$) contributions as follows
\begin{subequations} \label{eq: dod cell terms}
\begin{align}
\label{eq: dod cell term a0}
    \begin{split}
    J^{0, E}_h(u_h, w_h) & = \phantomminus \eta_E \sum_{(i, j) \in \IE^2, \, i < j} J^{0, E}_{h, ij} (u_h, w_h)\\
    & \phantomeq \binminus \eta_E \sum_{\gamma \in \GammaInt(E)} \int_{\gamma} \langle \tfrac{1}{2}[f_n(u^L_{\gamma}) + f_n(u^R_{\gamma})], \llbracket w_h \rrbracket \rangle\\
    & \phantomeq \binminus \eta_E \sum_{\gamma \in \GammaExt(E)} \int_{\gamma} \langle \tfrac{1}{2}[f_n(u_h) + f_n(\mathfrak{M}_n(u_h)], w_h \rangle,
    \end{split}\\
\label{eq: dod cell term a1}
    J^{1, E}_h(u_h, w_h) & = \phantomminus \eta_E \sum_{(i, j) \in \IE^2, \, i < j} J^{1, E}_{h, ij}(u_h, w_h),\\
\label{eq: dod cell term s}
    \begin{split}
    J^{s, E}_h(u_h, w_h) & = \phantomminus \eta_E \sum_{(i, j) \in \IE^2, \, i < j} J^{s, E}_{ij}(u_h, w_h)\\
    & \phantomeq \binminus \eta_E \sum_{\gamma \in \GammaInt(E)} \int_{\gamma} \langle S_n(u^L_{\gamma}, u^R_{\gamma}), \llbracket w_h \rrbracket \rangle\\
    & \phantomeq \binminus \eta_E \sum_{\gamma \in \GammaExt(E)} \int_{\gamma} \langle S_n(u_h, \mathfrak{M}_n(u_h), w_h \rangle.
    \end{split}
\end{align}
\end{subequations}

For the surface contributions $J_{h,ij}^{0,E}$ of $J_{h}^{0,E}$
we distinguish two cases: (1) two inner faces, i.e. $\GammaInt \in \GammaInt$ and $\gamma_j \in \GammaInt$, and (2) one face being a reflecting boundary face.

\begin{enumerate}[wide]
\item For two inner faces $\gamma_i, \gamma_j \in \GammaInt$ are associated with neighbor elements $E_i$ and $E_j$, the surface stabilization term between $E_i$ and $E_j$ is given by
\begin{equation}
\begin{split}
       J_{h,ij}^{0,E}(u_h, w_h) & = \phantomplus p_{ij}^E(\extop{E_i}(u_h), \extop{E_j}(u_h), \extop{E}(w_h) - \extop{E_j}(w_h))\\
& \phantomeq \binplus p_{ji}^E(\extop{E_i}(u_h), \extop{E_j}(u_h), \extop{E}(w_h) - \extop{E_i}(w_h)). 
\end{split}
\end{equation}

\item If one of the two face is a reflecting boundary face, the stabilization needs to be defined separately. W.l.o.g we assume that $\gamma_i \in \GammaExt$ and $\gamma_j \in \GammaInt$. The reflected surface stabilization term between $\gamma_i$ and $\gamma_j$ is then given by
\begin{equation}
\begin{split}
       J_{h,ij}^{0,E}(u_h, w_h) & = \phantomplus p_{ij}^E(\extop{E_j}^{\gamma_i}(u_h), \extop{E_j}(u_h), \extop{E}(w_h) - \extop{E_j}(w_h))\\
& \phantomeq \binplus p_{ji}^E(\extop{E_j}^{\gamma_i}(u_h), \extop{E_j}(u_h), \extop{E}(w_h)). 
\end{split}
\end{equation}
\end{enumerate}

To define the contribution $J_{h,ij}^{1,E}$ we introduce a weighting factor $\omega_{\mathcal{E}} := \{ -1 \text{ if } \mathcal{E} = E; \frac{1}{2} \text{ else}\}$. With this the volume stabilization is given as
\begin{equation} \label{eq: dod volume term ij}
\begin{split}
       J_{h,ij}^{1,E}(u_h, w_h) & = \phantomplus \sum_{\mathcal{E} \in \{E, E_i, E_j\}} \omega_{\mathcal{E}} \big[p_V^E (\extop{E_i}^{ij}(u_h), \extop{E_j}^{ij}(u_h), \extop{\mathcal{E}}^{ij}(w_h))\\
       & \phantomeq \phantomplus \hphantom{\sum_{\mathcal{E} \in \{E, E_i, E_j\}} \omega_{\mathcal{E}} \big[]}- \tfrac{2}{K(K-1)} \int_{E} \sum_{k=1}^d \langle f_k (\extop{\mathcal{E}}^{ij}(u_h)), \partial_k \extop{\mathcal{E}}^{ij}(w_h) \rangle \big]\\
& \phantomeq \binplus \: \sum_{\mathcal{E} \in \{E, E_i, E_j\}} \omega_{\mathcal{E}} p_V^{E,*} (\extop{E_i}^{ij}(w_h), \extop{E_j}^{ij}(w_h), \extop{\mathcal{E}}^{ij}(u_h)).
\end{split}
\end{equation}

Finally the contribution $J^{s, E}_{ij}$ to the dissipative stabilization is
\begin{equation} \label{eq: dod neighbor dissipation}
\begin{split}
J^{s, E}_{ij}(u_h, w_h) & = \phantomplus \tfrac{1}{6} \sum_{\gamma \in \Gamma_h(E) } \int_{\gamma} \langle S_n(\extop{E_i}^{ij}(u_h),\extop{E_j}^{ij}(u_h)), \extop{E_i}^{ij}(w_h) - \extop{E_j}^{ij}(w_h)) \rangle\\
& \phantomeq \binplus \tfrac{1}{6} \sum_{\gamma \in \Gamma_h(E)} \int_{\gamma} \langle S_n(\extop{E_j}^{ij}(u_h),\extop{E_i}^{ij}(u_h)), \extop{E_j}^{ij}(w_h) - \extop{E_i}^{ij}(w_h)) \rangle.
\end{split}
\end{equation}

\end{definition}
\begin{remark}[Structure of the stabilization terms]
We want to make a few remarks on the structure of the stabilization terms and comparisons with prior formulations.
\begin{enumerate}[wide]
\item 
The structure of the volume stabilization term~\eqref{eq: dod volume term ij} and that of $J^1_h$ in \cite{DoD_1d_nonlin_2022} are essentially identical, with two important differences. The obvious one is the use of the volume propagation forms and the weighting in front of the integral term (which already appeared in the central volume propagation forms). 
This is necessary to properly consider the relationships between multiple neighbor elements of the stabilized cut cell.
The second difference is
that we only use the central flux in~\eqref{eq: dod volume term ij}, respectively in our volume propagation forms, while in \cite{DoD_1d_nonlin_2022} the term $J^1_j$ contained arbitrary numerical fluxes that in general could contain dissipative fluxes.

Having no dissipative fluxes in the volume term has several important benefits.
First, it allows us to construct an energy-conservative formulation of the DoD method (or even a formulation with the SBP property \cite{petri2026kinetic}).
Secondly, in our experience the construction of a consistent extension of dissipative surface fluxes into the volume of the cut cell is very delicate, while adding dissipation purely via surface terms, as in~\eqref{eq: dod neighbor dissipation}, is straightforward.
Hence, the splitting between central and dissipative fluxes allows for a much simpler extension of our method to multiple space dimensions.
\item
The inclusion of a separate dissipation term like~\eqref{eq: dod neighbor dissipation} depending on extended polynomial functions was already done in \cite{fvca_proc_dod} where the amount of dissipation was based on eigendecompositions of the occuring operators. Here the construction is much simpler, due to the separation of central and dissipative terms.
\item
In~\eqref{eq: dod cell terms} the sums go over all index pairs $(i, j) \in \IE$ such that $i < j$ but since $J^{0, E}_{ij}$ contains the propagation form for both index pairs $(i, j)$ and $(j, i)$ we have an extended flux in both directions. What we do not have, though, is a flux for the index pair $(i, i)$, that is there is no recurring flux in the formulation. 
\item 
Even though the neighbor dissipation term~\eqref{eq: dod neighbor dissipation} is dissipative, the term~\eqref{eq: dod cell term s} is not. When proving energy stability we must consider $s_h + J^s_h$, which is indeed dissipative, as we will prove in section~\ref{sec: theoretical section}.
\end{enumerate}
\end{remark}

\section{Energy conservation and stability} \label{sec: theoretical section}

To prove statements about energy preservation, we need two additional properties. They are related to the concept of flux potential and entropy-conservative fluxes in the theory of entropy-conservative/stable schemes developed in \cite{TADMOR1987,Tadmor_2003}.

\begin{definition}[Energy-preserving propagation forms] \label{def: energy preserving propagation forms}
Let $E \in \cI$ be a small cut cell and consider a set of propagation forms consisting of $p_{ij}^E$ with $i, j \in \IE$, $i \neq j$, $p_V^E$ and $p_V^{E, *}$. We call this set of propagation forms energy preserving if it satisfies the following properties:
\begin{enumerate}
    \item \label{item: flux potential} Let $K = |\IE|$ be the number of faces of $E$. Then for all $w_h \in \mathcal{P}^r(E)^3$ it must hold that
    \begin{equation}
    \frac{2}{K(K-1)} \int_{E} f (w_h) \cdot \nabla w_h = \frac{1}{2}p_{ij}^E(w_h, w_h, w_h) + \frac{1}{2}p_{ji}^E(w_h, w_h, w_h).
    \end{equation}
    \item \label{item: energy conservative flux} Let $u^\mu, u^\nu \in \mathcal{P}^r(E)^3$. Then it must hold that
    \begin{equation}
p_{ij}^E(u^\mu, u^\nu, u^\mu - u^\nu) = \frac{1}{2}p_{ij}^E(u^\mu, u^\mu, u^\mu) - \frac{1}{2}p_{ij}^E(u^\nu, u^\nu, u^\nu).
\end{equation}
\end{enumerate}
\end{definition}
As an example we will show that the central propagation forms that we defined in~\ref{exam: central propagation forms} are energy preserving:

\begin{proposition} \label{prop: energy preservation central propagation forms}
Let $E \in \cI$ be a small cut cell. The set of central propagation forms on $E$ defined in example~\ref{exam: central propagation forms} is energy preserving.
\end{proposition}
\begin{proof}
\begin{enumerate}
    \item Let $w_h \in \mathcal{P}^r(E)^3$. We have
\begin{align*}
\int_{E} f (w_h) \cdot \nabla w_h & = \tfrac{1}{2}\int_{E} \sum_{k=1}^d \langle f_k (w_h), \partial_k w_h \rangle + \tfrac{1}{2} \int_{E} \sum_{k=1}^d \langle f_k (w_h), \partial_k w_h \rangle\\
& = \tfrac{1}{2}\int_{E} \sum_{k=1}^d \langle f_k (w_h), \partial_k w_h \rangle + \tfrac{1}{2} \int_{E} \sum_{k=1}^d \langle w_h, f_k(\partial_k w_h) \rangle\\
& = \tfrac{1}{2}\int_{E} \sum_{k=1}^d \langle f_k (w_h), \partial_k w_h \rangle + \tfrac{1}{2} \int_{E} \sum_{k=1}^d \langle w_h, \partial_k f_k(w_h) \rangle\\
& = \tfrac{1}{2} \int_E f(w_h) \cdot \nabla w_h + \tfrac{1}{2} \int_E [\nabla \cdot f(w_h)] w_h
\end{align*}
where the second equality is due to the symmetry of the linear operator $f_k$ and the third equality is due to $f_k$ being linear and constant in space. Multiplying with $\frac{2}{K(K-1)}$ and using~\eqref{eq: central volume propagation form} and~\eqref{eq: propagation balance} we get
\begin{align*}
\tfrac{2}{K(K-1)} \int_{E} f (w_h) \cdot \nabla w_h & = \tfrac{1}{2}\tfrac{2}{K(K-1)} \bigg[ \phantomplus\int_{E} \tfrac{1}{2}[f(w_h)+f(w_h)] \cdot \nabla w_h\\
& \phantomeq \hphantom{\tfrac{1}{2}\tfrac{2}{K(K-1)} \bigg[} + \int_E \tfrac{1}{2}[\nabla \cdot f(w_h) + \nabla \cdot f(w_h)] w_h \bigg] \\
& = \tfrac{1}{2}\big[ p_V^E(w_h, w_h, w_h) + p_V^{E, *}(w_h, w_h, w_h) \big]\\
& = \tfrac{1}{2} \big[ p_{ij}^E(w_h, w_h, w_h) + p_{ji}^{E}(w_h, w_h, w_h) \big].
\end{align*}
\item The form $p_{ij}^E$ is linear in the third argument. Furthermore by the definition of $p_{ij}^E$ and the central flux as well as the symmetry of $A_n$ (which implies $\langle A_n u^\mu, u^\nu \rangle = \langle u^\mu, A_n u^\nu \rangle)$ we have
\begin{align}
\label{eq: central propagation form splitting}
p_{ij}^E(u^\mu, u^\nu, w) & = \tfrac{1}{2} p_{ij}^E(u^\mu, u^\mu, w) + \tfrac{1}{2} p_{ij}^E(u^\nu, u^\nu, w)\\
\label{eq: central propagation form symmetry}
p_{ij}^E(u^\mu, u^\mu, u^\nu) & = p_{ij}^E(u^\nu, u^\nu, u^\mu)
\end{align}
for all $w \in \mathcal{P}^r(E)^3$ and all $u^\mu, u^\nu \in \mathcal{P}^r(E)^3$. Thus
\begin{align*}
p_{ij}^E(u^\mu, u^\nu, u^\mu - u^\nu) & = p_{ij}^E(u^\mu, u^\nu, u^\mu) - p_{ij}^E(u^\mu, u^\nu, u^\nu)\\
& \overset{\mathclap{\eqref{eq: central propagation form splitting}}}{=} \tfrac{1}{2}\big[p_{ij}^E(u^\mu, u^\mu, u^\mu) + p_{ij}^E(u^\nu, u^\nu, u^\mu)\big] - \tfrac{1}{2}\big[p_{ij}^E(u^\mu, u^\mu, u^\nu) + p_{ij}^E(u^\nu, u^\nu, u^\nu) \big]\\
& \overset{\mathclap{\eqref{eq: central propagation form symmetry}}}{=} \tfrac{1}{2}p_{ij}^E(u^\mu, u^\mu, u^\mu) - \tfrac{1}{2}p_{ij}^E(u^\nu, u^\nu, u^\nu).
\end{align*}
\end{enumerate}
\end{proof}
At the domain boundary additional complications arise and we need to take into account the reflecting wall boundary condition. For the base scheme~\eqref{eq: base scheme} energy conservation follows from the skew-symmetry of the reflecting wall boundary condition~\eqref{eq: boundary reflection skew symmetry} and we need to pose a similar condition for our propagation forms:
\begin{definition}[Reflecting propagation forms] \label{def: reflected propagation form}
Let $E \in \cI$ be a small cut cell and consider a set of propagation forms consisting of $p_{ij}^E$ with $i, j \in \IE, i \neq j$, $p_V^E$, and $p_V^{E, *}$. We call this set of propagation forms a set of reflecting propagation forms if for any $i \in \IE$, such that $\gamma_i \in \GammaExt(E)$ is a reflecting wall boundary face, there holds 
\begin{equation} \label{eq: reflected propagation form}
p_{ji}^E(u^\mu, \extop{E_j}^{\gamma_i}(u^\nu), w) = -p_{ji}^E(\extop{E_j}^{\gamma_i}(u^\mu), u^\nu, \extop{E_j}^{\gamma_i}(w)), \quad \forall j \in \IE, \; j \neq i,
\end{equation}
and all $u^\mu, u^\nu, w \in \mathcal{P}^r(E)^3$.
\end{definition}
Again this is only an abstract definition. Given assumption~\ref{item: reflecting wall face assumption} we can construct a set of reflecting propagation forms which essentially amounts to adding an element of the kernel of the linear system posed by~\eqref{eq: surface forms int by parts} and~\eqref{eq: propagation forms consistency} to the central propagation forms as follows:
\begin{example}[Reflecting central propagation forms] \label{exam: reflecting propagation forms}
Let $E \in \IE$ be a small cut cell and $\gamma_i \in \GammaExt(E)$, with $i \in \IE$ its index, a boundary face of $E$ with a reflecting wall boundary condition. Consider the forms
\[
\tilde{p}_{jk}^E(u^{\mu}, u^{\nu}, w) : (P^r(E))^3 \times (P^r(E))^3 \times (P^r(E))^3 \to \mathbb{R}
\]
where $j, k \in \IE$ and $j \neq k$ defined by
\begin{align*}
\tilde{p}_{ij}^E(u^{\mu}, u^{\nu}, w) & = - \tfrac{K-2}{K(K-1)} b_{\gamma_j}(u^\mu, u^\nu, w) + \tfrac{1}{K(K-1)} \sum_{k \in \IE, k \neq i, j} b_{\gamma_k}(u^\mu, u^\nu, w)\\
\tilde{p}_{ji}^E(u^{\mu}, u^{\nu}, w) & = \hphantom{-} \tfrac{K-2}{K(K-1)} b_{\gamma_j}(u^\mu, u^\nu, w) - \tfrac{1}{K(K-1)} \sum_{k \in \IE, k \neq i, j} b_{\gamma_k}(u^\mu, u^\nu, w)
\end{align*}
for $i \neq j \in \IE$ and
\begin{align*}
\tilde{p}_{jk}^E(u^{\mu}, u^{\nu}, w) & = \tfrac{1}{K(K-1)} b_{\gamma_k}(u^\mu, u^\nu, w) - \tfrac{1}{K(K-1)} b_{\gamma_j}(u^\mu, u^\nu, w)\\
\tilde{p}_{kj}^E(u^{\mu}, u^{\nu}, w) & = \tfrac{1}{K(K-1)} b_{\gamma_j}(u^\mu, u^\nu, w) - \tfrac{1}{K(K-1)} b_{\gamma_k}(u^\mu, u^\nu, w)
\end{align*}
for $j, k \in \IE$ and $j \neq k, k \neq i, j \neq i$. We want to show that the forms
\begin{equation} \label{eq: reflected central propagation forms}
p_{ij}^{E, \mathfrak{M}} = p_{ij}^{E, Z} + \tilde{p}_{ij}, i, j \in \IE, i \neq j, \quad p_V^{E, Z} \text{ and } p_V^{E, Z, *}
\end{equation}
constitute a set of propagation forms. First, we note that
\[
\tilde{p}^E_{jk} + \tilde{p}^E_{kj} = 0 \quad \forall j, k \in \IE, j \neq k,
\]
even if $i = j$ or $i = k$.
Furthermore, we have
\begin{align*}
\sum_{j \in \IE, j \neq i} \tilde{p}^E_{ji} & = \tfrac{K-2}{K(K-1)} \sum_{j \in \IE, j \neq i} b_{\gamma_j}(u^\mu, u^\nu, w) - \tfrac{1}{K(K-1)} \sum_{j \in \IE, j \neq i} \sum_{k \in \IE, k \neq i, j} b_{\gamma_k}(u^\mu, u^\nu, w)\\
& = \tfrac{K-2}{K(K-1)} \sum_{j \in \IE, j \neq i} b_{\gamma_j}(u^\mu, u^\nu, w) - \tfrac{K-2}{K(K-1)} \sum_{j \in \IE, j \neq i}  b_{\gamma_j}(u^\mu, u^\nu, w)\\
& = 0
\end{align*}
and also, if $i \neq j \in \IE$,
\begin{align*}
    \sum_{k \in \IE, k \neq j} \tilde{p}^E_{kj} & = \hphantom{-} \sum_{k \in \IE, k \neq j, i} \tilde{p}^E_{kj} + \tilde{p}_{ij}^E\\
    & = \hphantom{-} \tfrac{1}{K(K-1)} \sum_{k \in \IE, k \neq j, i} b_{\gamma_j}(u^\mu, u^\nu, w) - \tfrac{1}{K(K-1)} \sum_{k \in \IE, k \neq j, i} b_{\gamma_k}(u^\mu, u^\nu, w)\\
    & \phantomeq - \tfrac{K-2}{K(K-1)} b_{\gamma_j}(u^\mu, u^\nu, w) + \tfrac{1}{K(K-1)} \sum_{k \in \IE, k \neq i, j} b_{\gamma_k}(u^\mu, u^\nu, w)\\
    & = 0.
\end{align*}
From this it follows that the forms~\eqref{eq: reflected central propagation forms} satisfy the properties~\eqref{eq: propagation balance} and~\eqref{eq: propagation forms consistency}. Let us now look at~\eqref{eq: reflected propagation form}. We have by construction
\begin{align*}
p^{E, \mathfrak{M}}_{ji}(u^\mu, \extop{E_j}^{\gamma_i}(u^\nu), w) & = \tfrac{1}{K-1} b_{\gamma_i}(u^\mu, \extop{E_j}^{\gamma_i}(u^\nu), w)\\
& \overset{\eqref{eq: surface forms}}{=} \tfrac{1}{K-1} \int_{\gamma_i} \langle \tfrac{1}{2}\big [f_n(u^\mu) + f_n(\mathfrak{M}_n(u^\nu)) \big], w \rangle\\
& \overset{\eqref{eq: reflecting wall boundary operator}}{=} -\tfrac{1}{K-1} \int_{\gamma_i} \langle \tfrac{1}{2}\big [f_n(\mathfrak{M}_n(u^\mu)) + f_n(u^\nu) \big], \mathfrak{M}_n(w) \rangle\\
& = -p^{E, \mathfrak{M}}_{ji}(\extop{E_j}^{\gamma_i}(u^\mu), u^\nu, \extop{E_j}^{\gamma_i}(w))
\end{align*}
so \eqref{eq: reflected propagation form} is satisfied. 

Finally we note that these reflecting propagation forms are also energy preserving. Indeed, looking again at the proof of proposition~\ref{prop: energy preservation central propagation forms} we note that we only relied on the particular choice of volume propagation forms, which we did not change here and the usage of the central flux, which did not change either. Thus the same proof can be applied to the reflecting propagation forms~\eqref{eq: reflected central propagation forms}, i. e. they are energy preserving.
\end{example}
We will now prove that the central stabilization terms~\eqref{eq: dod cell term a0} and~\eqref{eq: dod cell term a1} conserve energy, under the assumption that the underlying set of propagation forms is energy preserving, i. e. it satisfies definition \ref{def: energy preserving propagation forms} and also definition \ref{def: reflected propagation form} in case a reflecting wall boundary condition is applied:
\begin{proposition}[Energy conservation] \label{prop: l2-conservation}
Let $E \in \cI$ be a small cut cell. Assume that the underlying set of propagation forms of~\eqref{eq: dod cell term a0} and~\eqref{eq: dod cell term a1} is energy preserving and a set of reflected propagation forms. Then it holds that
\begin{equation}
    J^{E, 0}_{h}(u_h, u_h) + J^{E, 1}_{h}(u_h, u_h) = 0,
\end{equation}
in other words, the central stabilization forms conserve the energy.
\end{proposition}
\begin{proof}
The key in the proof will be the investigation of the pairwise interaction between two faces $\gamma_i$ and $\gamma_j$ with $i, j \in \IE$. The handling of such face pairs will differ depending on whether one of these faces is a reflecting wall boundary or not. If this is the case, w.\,l.\,o.\,g. we will always assume that $\gamma_i$ is the reflecting face, to ease notation.

Using linearity in the third argument, we begin by splitting the central surface stabilization into
\[
J^{0, E}_{h, ij}(u_h, u_h) = T_{ij}^1 + T_{ij}^2
\]
where
\begin{enumerate}
\item for two internal faces we set
\begin{align*}
T_{ij}^1 & = -p_{ij}^E(\extop{E_i}(u_h), \extop{E_j}(u_h), \extop{E_j}(u_h)) - p_{ji}^E(\extop{E_i}(u_h), \extop{E_j}(u_h), \extop{E_i}(u_h)),\\
T_{ij}^2 & = p_{ij}^E(\extop{E_i}(u_h), \extop{E_j}(u_h), \extop{E}(u_h)) + p_{ji}^E(\extop{E_i}(u_h), \extop{E_j}(u_h), \extop{E}(u_h)),
\end{align*}
\item for $\gamma_i$ a boundary face we set
\begin{align*}
T_{ij}^1 & = -p_{ij}^E(\extop{E_j}^{\gamma_i}(u_h), \extop{E_j}(u_h), \extop{E_j}(u_h)),\\
T_{ij}^2 & = p_{ij}^E(\extop{E_j}^{\gamma_i}(u_h), \extop{E_j}(u_h), \extop{E}(u_h)) + p_{ji}^E(\extop{E_j}^{\gamma_i}(u_h), \extop{E_j}(u_h), \extop{E}(u_h)).
\end{align*}
\end{enumerate}
Now we use part~\ref{lem: propagation forms consistency} of definition~\ref{def: propagation forms} to write
\begin{align*}
J^{0, E}_h(u_h, u_h) & = \phantomminus \eta_E \sum_{(i, j) \in \IE^2, \, i < j} T_{ij}^1 + T_{ij}^2 - \eta_E \sum_{\gamma \in \GammaInt(E)} \int_{\gamma} \langle \tfrac{1}{2}[f_n(u^L_{\gamma}) + f_n(u^R_{\gamma})], \llbracket u_h \rrbracket \rangle\\
& \phantomeq \binminus \eta_E \sum_{\gamma \in \GammaExt(E)} \int_{\gamma} \langle \tfrac{1}{2}[f_n(u_h) + f_n(\mathfrak{M}_n(u_h))], u_h \rangle\\
& = \phantomminus \eta_E \sum_{(i, j) \in \IE^2, \, i < j} T_{ij}^1 + T_{ij}^2 - \eta_E \sum_{\gamma \in \GammaInt(E)} b_{\gamma} (\extop{E^L_\gamma}(u_h), \extop{E^R_\gamma}(u_h), \extop{E^L_\gamma}(u_h) - \extop{E^R_\gamma}(u_h))\\
& \phantomeq \binminus \eta_E \sum_{\gamma \in \GammaExt(E)} b_{\gamma} (\extop{E}(u_h), \extop{E}^\gamma(u_h), \extop{E}(u_h)\\
& = \eta_E \sum_{(i, j) \in \IE^2, \, i < j} T_{ij}^1 + T_{ij}^2 + \eta_E \sum_{(i, j) \in \IE^2, \, i < j} T_{ij}^3
\end{align*}
with,
\begin{enumerate}
    \item in case of two internal faces,
\[
T_{ij}^3 = -p_{ij}^E \big(\extop{E}(u_h), \extop{E_j}(u_h), \extop{E}(u_h) - \extop{E_j}(u_h) \big) - p_{ji}^E \big(\extop{E}(u_h), \extop{E_i}(u_h), \extop{E}(u_h) - \extop{E_i}(u_h) \big)
\]
\item if $\gamma_i$ is a boundary face,
\[
T_{ij}^3 = -p_{ij}^E \big(\extop{E}(u_h), \extop{E_j}(u_h), \extop{E}(u_h) - \extop{E_j}(u_h) \big) - p_{ji}^E \big(\extop{E}(u_h), \extop{E}^{\gamma_i}(u_h), \extop{E}(u_h) \big).
\]
\end{enumerate}
We also split the central volume stabilization into
\[
J^{1, E}_{h, ij}(u_h, u_h) = \eta_E \big(T_{ij}^{V, 1} + T_{ij}^{V, 2}\big)
\]
with (splitting the sums into contributions from $E_i$ and $E_j$, and contributions from $E$, respectively)
\begin{align*}
T_{ij}^{V, 1} & = \phantomplus \sum_{\mathcal{E} \in \{E_i, E_j\}} \omega_{\mathcal{E}} \big[p_V^E (\extop{E_i}^{ij}(u_h), \extop{E_j}^{ij}(u_h), \extop{\mathcal{E}}^{ij}(u_h)) - \tfrac{2}{K(K-1)} \int_{E}  f (\extop{\mathcal{E}}^{ij}(u_h)) \cdot \nabla \extop{\mathcal{E}}^{ij}(u_h)  \big]\\
& \phantomeq \binplus \: \sum_{\mathcal{E} \in \{E_i, E_j\}} \omega_{\mathcal{E}} p_V^{E,*} (\extop{E_i}^{ij}(u_h), \extop{E_j}^{ij}(u_h), \extop{\mathcal{E}}^{ij}(u_h))\\
& = \phantomplus \sum_{\mathcal{E} \in \{E_i, E_j\}} \tfrac{1}{2} \big[p_V^E (\extop{E_i}^{ij}(u_h), \extop{E_j}^{ij}(u_h), \extop{\mathcal{E}}^{ij}(u_h)) + p_V^{E,*} (\extop{E_i}^{ij}(u_h), \extop{E_j}^{ij}(u_h), \extop{\mathcal{E}}^{ij}(u_h)) \big]\\
& \phantomeq  \: -\sum_{\mathcal{E} \in \{E_i, E_j\}}  \tfrac{1}{2}\tfrac{2}{K(K-1)} \int_{E} f (\extop{\mathcal{E}}^{ij}(u_h)) \cdot \nabla \extop{\mathcal{E}}^{ij}(u_h)
\end{align*}
and
\begin{align*}
T_{ij}^{V, 2} & = \phantomplus \omega_{E} \big[p_V^E (\extop{E_i}^{ij}(u_h), \extop{E_j}^{ij}(u_h), \extop{E}(u_h)) - \tfrac{2}{K(K-1)} \int_{E}  f (\extop{E}(u_h)) \cdot \nabla \extop{E}(u_h) \big]\\
& \phantomeq \binplus \omega_{E} p_V^{E,*} (\extop{E_i}^{ij}(u_h), \extop{E_j}^{ij}(u_h), \extop{E}(u_h))\\
& = -p_V^E (\extop{E_i}^{ij}(u_h), \extop{E_j}^{ij}(u_h), \extop{E}(u_h)) + \tfrac{2}{K(K-1)} \int_{E}  f (\extop{E}(u_h)) \cdot \nabla \extop{E}(u_h)\\
& \phantomeq - p_V^{E,*} (\extop{E_i}^{ij}(u_h), \extop{E_j}^{ij}(u_h), \extop{E}(u_h))
\end{align*}
where we used $\omega_{\mathcal{E}} = \{ -1 \text{ if } \mathcal{E} = E; \frac{1}{2} \text{ else}\}$. Using~\eqref{eq: propagation balance} we rewrite $T_{ij}^{V, 1}$ as
\begin{align*}
T_{ij}^{V, 1} & = \phantomplus \sum_{\mathcal{E} \in \{E_i, E_j\}} \tfrac{1}{2} \big[p_{ij}^E (\extop{E_i}^{ij}(u_h), \extop{E_j}^{ij}(u_h), \extop{\mathcal{E}}^{ij}(u_h)) + p_{ji}^{E} (\extop{E_i}^{ij}(u_h), \extop{E_j}^{ij}(u_h), \extop{\mathcal{E}}^{ij}(u_h)) \big]\\
& \phantomeq  \: -\sum_{\mathcal{E} \in \{E_i, E_j\}}  \tfrac{1}{2}\tfrac{2}{K(K-1)} \int_{E} f (\extop{\mathcal{E}}^{ij}(u_h)) \cdot \nabla \extop{\mathcal{E}}^{ij}(u_h) 
\end{align*}
and similarly $T_{ij}^{V, 2}$ as
\begin{align*}
T_{ij}^{V, 2} & = - p^{E}_{ij}(\extop{E_i}^{ij}(u_h), \extop{E_j}^{ij}(u_h), \extop{E}(u_h)) - p^{E}_{ji}(\extop{E_i}^{ij}(u_h), \extop{E_j}^{ij}(u_h), \extop{E}(u_h))\\
& \phantomeq \binplus \tfrac{2}{K(K-1)} \int_{E} f (\extop{E}(u_h)) \cdot \nabla \extop{E}(u_h) \rangle\\
& \overset{\mathclap{\eqref{eq: unified extension operator}}}{=} - T_{ij}^2 + \tfrac{2}{K(K-1)} \int_{E}  f (\extop{E}(u_h)) \cdot \nabla \extop{E}(u_h)
\end{align*}
Furthermore, we have
\begin{enumerate}
    \item if both faces are internal faces,
\begin{align*}
T_{ij}^1 + T_{ij}^{V, 1} & \overset{\mathclap{\eqref{eq: unified extension operator}}}{=} \phantomplus \tfrac{1}{2} \big[ \phantomplus p_{ij}^E (\extop{E_i}(u_h), \extop{E_j}(u_h), \extop{E_i}(u_h) - \extop{E_j}(u_h))\\
& \phantomeq \phantomminus \hphantom{\tfrac{1}{2} \big[}+ p_{ji}^{E} (\extop{E_i}(u_h), \extop{E_j}(u_h), \extop{E_j}(u_h) - \extop{E_i}(u_h)) \big]\\
& \phantomeq  \: -\sum_{\mathcal{E} \in \{E_i, E_j\}}  \tfrac{1}{2}\tfrac{2}{K(K-1)} \int_{E} f (\extop{\mathcal{E}}^{ij}(u_h)) \cdot \nabla \extop{\mathcal{E}}^{ij}(u_h)
\end{align*}
\item if $\gamma_i$ is a boundary face,
\begin{align*}
T_{ij}^1 + T_{ij}^{V, 1} & \overset{\mathclap{\eqref{eq: unified extension operator}}}{=} \phantomplus \tfrac{1}{2} \big[ \phantomplus p_{ij}^E (\extop{E_j}^{\gamma_i}(u_h), \extop{E_j}(u_h), \extop{E_j}^{\gamma_i}(u_h) - \extop{E_j}(u_h))\\
& \phantomeq \phantomminus \hphantom{\tfrac{1}{2} \big[}+ p_{ji}^{E} (\extop{E_j}^{\gamma_i}(u_h), \extop{E_j}(u_h), \extop{E_j}(u_h) + \extop{E_j}^{\gamma_i}(u_h)) \big]\\
& \phantomeq  \: -\sum_{\mathcal{E} \in \{E_i, E_j\}}  \tfrac{1}{2}\tfrac{2}{K(K-1)} \int_{E} f (\extop{\mathcal{E}}^{ij}(u_h)) \cdot \nabla \extop{\mathcal{E}}^{ij}(u_h) \\
& \overset{\mathclap{\eqref{eq: reflected propagation form}}}{=} \phantomplus \tfrac{1}{2} \big[p_{ij}^E (\extop{E_j}^{\gamma_i}(u_h), \extop{E_j}(u_h), \extop{E_j}^{\gamma_i}(u_h) - \extop{E_j}(u_h)) \big]\\
& \phantomeq  \: -\sum_{\mathcal{E} \in \{E_i, E_j\}}  \tfrac{1}{2}\tfrac{2}{K(K-1)} \int_{E} f (\extop{\mathcal{E}}^{ij}(u_h)) \cdot \nabla \extop{\mathcal{E}}^{ij}(u_h)
\end{align*}
where we used the definition~\ref{def: reflected propagation form} of reflected propagation forms.
\end{enumerate}
Using property~\ref{item: flux potential} of definition \ref{def: energy preserving propagation forms} we rewrite $T_{ij}^{V,2}$ as
\begin{align} \label{eq: TVij2 rewrite}
T_{ij}^{V, 2} & = - T_{ij}^2 + \tfrac{1}{2} \big[ p_{ij}^E(\extop{E}(u_h), \extop{E}(u_h), \extop{E}(u_h)) + p_{ji}^E(\extop{E}(u_h), \extop{E}(u_h), \extop{E}(u_h)) \big]
\end{align}
and $T_{ij}^1 + T_{ij}^{V, 1}$ either,
\begin{enumerate}
    \item if both faces are internal, as
\begin{align*}
T_{ij}^1 + T_{ij}^{V, 1} & = \phantomplus \tfrac{1}{2} \big[ \phantomplus p_{ij}^E (\extop{E_i}(u_h), \extop{E_j}(u_h), \extop{E_i}(u_h) - \extop{E_j}(u_h))\\
& \phantomeq \phantomminus \hphantom{\tfrac{1}{2} \big[}\binplus p_{ji}^{E} (\extop{E_i}(u_h), \extop{E_j}(u_h), \extop{E_j}(u_h) - \extop{E_i}(u_h)) \big]\\
& \phantomeq  \: - \tfrac{1}{2 }\sum_{\mathcal{E} \in \{E_i, E_j\}}  \tfrac{1}{2} \big [p_{ij}^E(\extop{\mathcal{E}}^{ij}(u_h), \extop{\mathcal{E}}^{ij}(u_h), \extop{\mathcal{E}}^{ij}(u_h)) + p_{ji}^E(\extop{\mathcal{E}}^{ij}(u_h), \extop{\mathcal{E}}^{ij}(u_h), \extop{\mathcal{E}}^{ij}(u_h))]
\end{align*}
\item or if $\gamma_i$ is a boundary face, as
\begin{align*}
T_{ij}^1 + T_{ij}^{V, 1} & = \phantomplus \tfrac{1}{2} \big[p_{ij}^E (\extop{E_j}^{\gamma_i}(u_h), \extop{E_j}(u_h), \extop{E_j}^{\gamma_i}(u_h) - \extop{E_j}(u_h)) \big]\\
& \phantomeq  \: - \tfrac{1}{2 }\sum_{\mathcal{E} \in \{E_i, E_j\}}  \tfrac{1}{2} \big [p_{ij}^E(\extop{\mathcal{E}}^{ij}(u_h), \extop{\mathcal{E}}^{ij}(u_h), \extop{\mathcal{E}}^{ij}(u_h)) + p_{ji}^E(\extop{\mathcal{E}}^{ij}(u_h), \extop{\mathcal{E}}^{ij}(u_h), \extop{\mathcal{E}}^{ij}(u_h))]
\end{align*}
\end{enumerate}
We introduce another term $T_{ij}^4$ that takes the form,
\begin{enumerate}
    \item if both faces are internal faces,
\[
T_{ij}^4 = - \tfrac{1}{2} \big[ p_{ij}^E(\extop{E_j}(u_h), \extop{E_j}(u_h), \extop{E_j}(u_h)) + p_{ji}^E(\extop{E_i}(u_h), \extop{E_i}(u_h), \extop{E_i}(u_h)) \big].
\]
\item if $\gamma_i$ is a boundary face,
\[
T_{ij}^4 = - \tfrac{1}{2} p_{ij}^E(\extop{E_j}(u_h), \extop{E_j}(u_h), \extop{E_j}(u_h)).
\]
\end{enumerate}
Then we have, using \eqref{eq: TVij2 rewrite} and property~\ref{item: energy conservative flux} of definition~\ref{def: energy preserving propagation forms}, that
\begin{enumerate}
    \item if both faces are internal,
\begin{align*}
T_{ij}^2 + T_{ij}^{V, 2} + T_{ij}^3 + T_{ij}^4 & = -p_{ij}^E(\extop{E}(u_h), \extop{E_j}(u_h), \extop{E}(u_h) - \extop{E_j}(u_h))\\
& \phantomeq - p_{ji}^E(\extop{E}(u_h), \extop{E_i}(u_h), \extop{E}(u_h) - \extop{E_i}(u_h))\\
& \phantomeq + \tfrac{1}{2} \big[ p_{ij}^E(\extop{E}(u_h), \extop{E}(u_h), \extop{E}(u_h)) + p_{ji}^E(\extop{E}(u_h), \extop{E}(u_h), \extop{E}(u_h)) \big]\\
& \phantomeq - \tfrac{1}{2} \big[ p_{ij}^E(\extop{E_j}(u_h), \extop{E_j}(u_h), \extop{E_j}(u_h)) + p_{ji}^E(\extop{E_i}(u_h), \extop{E_i}(u_h), \extop{E_i}(u_h)) \big]\\
& = 0.
\end{align*}
\item if $\gamma_i$ is a boundary face,
\begin{align*}
T_{ij}^2 + T_{ij}^{V, 2} + T_{ij}^3 + T_{ij}^4 & = -p_{ij}^E(\extop{E}(u_h), \extop{E_j}(u_h), \extop{E}(u_h) - \extop{E_j}(u_h)) - p_{ji}^E(\extop{E}(u_h), \extop{E}^{\gamma_i}(u_h), \extop{E}(u_h))\\
& \phantomeq + \tfrac{1}{2} \big[ p_{ij}^E(\extop{E}(u_h), \extop{E}(u_h), \extop{E}(u_h)) + p_{ji}^E(\extop{E}(u_h), \extop{E}(u_h), \extop{E}(u_h)) \big]\\
& \phantomeq - \tfrac{1}{2} p_{ij}^E(\extop{E_j}(u_h), \extop{E_j}(u_h), \extop{E_j}(u_h))\\
& = 0.
\end{align*}
\end{enumerate}
Using property~\ref{item: energy conservative flux} of definition~\ref{def: energy preserving propagation forms} as well as \eqref{eq: unified extension operator} again we also have,
\begin{enumerate}
    \item if both faces are internal faces,
\begin{align*}
T_{ij}^1 + T_{ij}^{V, 1} - T_{ij}^4 & = \phantomplus \tfrac{1}{2} \big[ \phantomplus p_{ij}^E (\extop{E_i}(u_h), \extop{E_j}(u_h), \extop{E_i}(u_h) - \extop{E_j}(u_h))\\
& \phantomeq \phantomplus \hphantom{\tfrac{1}{2}\big[} + p_{ji}^{E} (\extop{E_i}(u_h), \extop{E_j}(u_h), \extop{E_j}(u_h) - \extop{E_i}(u_h)) \big]\\
& \phantomeq \binminus \tfrac{1}{2 }\sum_{\mathcal{E} \in \{E_i, E_j\}}  \tfrac{1}{2} \big [p_{ij}^E(\extop{\mathcal{E}}(u_h), \extop{\mathcal{E}}(u_h), \extop{\mathcal{E}}(u_h)) + p_{ji}^E(\extop{\mathcal{E}}(u_h), \extop{\mathcal{E}}(u_h), \extop{\mathcal{E}}(u_h))]\\
& \phantomeq \binplus \tfrac{1}{2} \big[ p_{ij}^E(\extop{E_j}(u_h), \extop{E_j}(u_h), \extop{E_j}(u_h)) + p_{ji}^E(\extop{E_i}(u_h), \extop{E_i}(u_h), \extop{E_i}(u_h)) \big]\\
& = 0.
\end{align*}
\item if $\gamma_i$ is a boundary face,
\begin{align*}
T_{ij}^1 + T_{ij}^{V, 1} - T_{ij}^4 & = \phantomplus \tfrac{1}{2} \big[p_{ij}^E (\extop{E_j}^{\gamma_i}(u_h), \extop{E_j}(u_h), \extop{E_j}^{\gamma_i}(u_h) - \extop{E_j}(u_h)) \big]\\
& \phantomeq \binminus \tfrac{1}{2 }\sum_{\mathcal{E} \in \{E_i, E_j\}}  \tfrac{1}{2} \big [p_{ij}^E(\extop{\mathcal{E}}^{ij}(u_h), \extop{\mathcal{E}}^{ij}(u_h), \extop{\mathcal{E}}^{ij}(u_h)) + p_{ji}^E(\extop{\mathcal{E}}^{ij}(u_h), \extop{\mathcal{E}}^{ij}(u_h), \extop{\mathcal{E}}^{ij}(u_h))]\\
& \phantomeq \binplus \tfrac{1}{2}  p_{ij}^E(\extop{E_j}(u_h), \extop{E_j}(u_h), \extop{E_j}(u_h))\\
& \overset{\mathclap{\eqref{eq: reflected propagation form}}}{=} \phantomplus \tfrac{1}{2} \big[p_{ij}^E (\extop{E_j}^{\gamma_i}(u_h), \extop{E_j}(u_h), \extop{E_j}^{\gamma_i}(u_h) - \extop{E_j}(u_h)) \big]\\
& \phantomeq \binminus \tfrac{1}{2 }\sum_{\mathcal{E} \in \{E_i, E_j\}}  \tfrac{1}{2} p_{ij}^E(\extop{\mathcal{E}}^{ij}(u_h), \extop{\mathcal{E}}^{ij}(u_h), \extop{\mathcal{E}}^{ij}(u_h))\\
& \phantomeq \binplus \tfrac{1}{2}  p_{ij}^E(\extop{E_j}(u_h), \extop{E_j}(u_h), \extop{E_j}(u_h))\\
& = 0.
\end{align*}
\end{enumerate}
If there is no face $\gamma_i$ with $i \in \IE$ and $\gamma_i \in \GammaExt(E)$ we are done because in this case we have
\[
J^{0, E}_h(u_h, u_h) + J^{1, E}_h(u_h, u_h) = \eta_E \big(\sum_{(i, j) \in \IE^2, \, i < j} T_{ij}^1 + T_{ij}^2 + T_{ij}^3 + T_{ij}^4 - T_{ij}^4 + T_{ij}^{V, 1} + T_{ij}^{V, 2}\big) = 0.
\]
Otherwise if there is a face (by assumption single) $\gamma_i \in \GammaExt(E)$ we still have to handle
\begin{align*}
J^{0, E}_h(u_h, u_h) + J^{1, E}_h(u_h, u_h) & = \eta_E \big(\sum_{(i, j) \in \IE^2, \, i < j} T_{ij}^1 + T_{ij}^2 + T_{ij}^3 + T_{ij}^4 - T_{ij}^4 + T_{ij}^{V, 1} + T_{ij}^{V, 2}\big)\\
& = \eta \sum_{j \in \IE, j \neq i} \big[p_{ji}^E(\extop{E}(u_h), \extop{E}^{\gamma_i}(u_h), \extop{E}(u_h))\\
& \phantomeq \hphantom{\eta \sum_{j \in \IE, j \neq i} \big[} -\tfrac{1}{2}p_{ji}^E(\extop{E}(u_h), \extop{E}(u_h), \extop{E}(u_h))\big].
\end{align*}
Using property \eqref{eq: propagation forms consistency} of definition~\ref{def: propagation forms} we have
\begin{align*}
    & \phantomeq \eta \sum_{j \in \IE, j \neq i} \big[p_{ji}^E(\extop{E}(u_h), \extop{E}^{\gamma_i}(u_h), \extop{E}(u_h)) -\tfrac{1}{2}p_{ji}^E(\extop{E}(u_h), \extop{E}(u_h), \extop{E}(u_h))\big]\\
    & = \eta \big[b_{\gamma_i}(\extop{E}(u_h), \extop{E}^{\gamma_i}(u_h), \extop{E}(u_h)) - \tfrac{1}{2} b_{\gamma_i}(\extop{E}(u_h), \extop{E}(u_h), \extop{E}(u_h)) \big]\\
    & \overset{\mathclap{\eqref{eq: boundary reflection skew symmetry}}}{=} \eta \big[\tfrac{1}{2}\int_{\gamma_i}\langle A_n u_h, u_h \rangle - \tfrac{1}{2} b_{\gamma_i}(\extop{E}(u_h), \extop{E}(u_h), \extop{E}(u_h)) \big] = 0
\end{align*}
and thus the proof is complete.
\end{proof}
The above proposition ensures energy conservation for the central stabilization term. If $S_n = 0$, this will already be enough to show energy conservation. However, if we want to add dissipation, we would like to have that the dissipative stabilization term $J^s_h$ is non-negative to have an energy dissipative scheme. Unfortunately this is not the case, so we need to look for a different strategy. The trick will be to compensate non-dissipative terms from $J^s_h$ with terms from $s_h$, ensuring that their sum is dissipative. This strategy was already used in \cite{DoD_1d_nonlin_2022} and \cite{dod_icosahom}, although in these works there was no explicit splitting between conservative and dissipative terms.
\begin{lemma}\label{lem: global dissipation}
Let $w_h \in \dspace$ be a discrete function. Then
\begin{equation}
s_h(w_h, w_h) + J^s_h(w_h, w_h) \geq 0.
\end{equation}
\end{lemma}
\begin{proof}
By definition $s_h(w_h, w_h) \geq 0$ since $S_n^M$ is positive semi-definite. Considering
\begin{align*}
J^{s, E}_h(w_h, w_h) & = \phantomminus \eta_E \sum_{(i, j) \in \IE^2, \, i < j} J^{s, E}_{ij}(w_h, w_h)  -\eta_E \sum_{\gamma \in \GammaInt(E) }  \int_{\gamma} \langle S_n(w^L_{\gamma}, w^R_{\gamma}), \llbracket w_h \rrbracket \rangle\\
& \phantomeq \binminus \eta_E \sum_{\gamma \in \GammaExt(E) }  \int_{\gamma} \langle S_n(w_h, \mathfrak{M}_n(w_h)), w_h \rangle
\end{align*}
we note that the first sum can be bounded from below since $J^{s, E}_{ij}(w_h, w_h) \geq 0$, which follows from~\eqref{eq: dod neighbor dissipation} and again that $S_n^M$ is positive semi-definite. We are thus left with the sum
\[
-\sum_{E \in \cI} \bigg [ \eta_E \sum_{\gamma \in \GammaInt(E) }  \int_{\gamma} \langle S_n(w^L_{\gamma}, w^R_{\gamma}), \llbracket w_h \rrbracket \rangle + \eta_E \sum_{\gamma \in \GammaExt(E) }  \int_{\gamma} \langle S_n(w_h, \mathfrak{M}_n(w_h)), w_h \rangle \bigg]
\]
which are indeed negative semi-definite. Since small cells do not have small neighbors, any face $\gamma$ in the above sum appears only once. Furthermore, the integrand $\langle S_n(w^L_{\gamma}, w^R_{\gamma}), \llbracket w_h \rrbracket \rangle$ is precisely the same as in~\eqref{eq: base scheme dissipative term}. Since $\eta_E \in [0, 1]$ we get that
\[
s_h(w_h, w_h) -\sum_{E \in \cI} \bigg [ \eta_E \sum_{\gamma \in \GammaInt(E) }  \int_{\gamma} \langle S_n(w^L_{\gamma}, w^R_{\gamma}), \llbracket w_h \rrbracket \rangle + \eta_E \sum_{\gamma \in \GammaExt(E) }  \int_{\gamma} \langle S_n(w_h, \mathfrak{M}_n(w_h)), w_h \rangle \bigg] \geq 0
\]
which finishes the proof.
\end{proof}

\begin{theorem}
We have for the approximate solution $u_h(t)$ and for any $T > 0$ that
\begin{equation}\label{eq: energy theorem}
\frac{1}{2} ||u_h(T)||^2_{L^2(\Omega)} \leq \frac{1}{2} ||u_h(0)||^2_{L^2(\Omega)}
\end{equation}
with equality if $S_n = 0$, i.e. if there is no dissipation.
\end{theorem}
\begin{proof}
Testing with $u_h(t)$ in~\eqref{eq: stabilized scheme} gives
\[
\frac{1}{2} \partial_t ||u_h(t)||_{L^2(\Omega)}^2 + a_h(u_h(t), u_h(t)) + J^a_h(u_h(t), u_h(t)) + s_h(u_h(t), u_h(t)) + J^s_h(u_h(t), u_h(t)) = 0.
\]
Integration over $(0, T)$ yields
\begin{equation*}
\begin{split}
\frac{1}{2} ||u_h(T)||^2_{L^2(\Omega)} & = \frac{1}{2} ||u_h(0)||^2_{L^2(\Omega)} - \int_0^T \phantomplus a_h(u_h(t), u_h(t)) + J^a_h(u_h(t), u_h(t))\\
& \phantomeq \hphantom{\frac{1}{2} ||u_h(0)||^2_{L^2(\Omega)} - \int_0^T} + s_h(u_h(t), u_h(t)) + J^s_h(u_h(t), u_h(t)).
\end{split}
\end{equation*}
It is a known result that $a_h(u_h(t), u_h(t))$ vanishes and due to proposition~\ref{prop: l2-conservation} the central stabilization term $J^a_h(u_h(t), u_h(t))$ also vanishes. We are thus left with the dissipative terms
\[
s_h(u_h(t), u_h(t)) + J^s_h(u_h(t), u_h(t))
\]
which we know to be non-negative from lemma~\ref{lem: global dissipation}. If $S_n$ is zero, they vanish and we have equality.
\end{proof}

\section{Numerical results} \label{sec: numerical section}

We present numerical results to demonstrate the stability and accuracy properties of our scheme. 

\subsection{Implementation}
We implement our numerical tests using the PDE framework DUNE \cite{dune-recent, dunepaperI:08, dunepaperII:08, dune-functions-1, dune-functions-2}. The cut-cell construction builds upon the TPMC~\cite{tpmc} library and parallel execution of test cases uses GNU Parallel~\cite{tange_gnu_parallel}.
For reproducibility the complete implementation with the main dependencies and build instructions can be found under~\cite{zenodo-code}.

For numerical simulations we need to pick a time step size and a threshold for the small cell classification~\eqref{eq: set of stab cut cells}. As already mentioned our goal is to solve the \textit{small cell problem}, i.e., we want to choose a time step size based on background cells even in the presence of arbitrarily small cut cells. Let $h$ denote the face length of a cartesian background element. A typical CFL condition for a DG scheme solving the wave equation would be
\begin{equation}\label{eq: standard cfl condition}
    \Delta t \leq \frac{1}{2r + 1}\frac{h}{c}
\end{equation}
where $c$ is the speed of sound. Using this CFL condition would essentially require us to stabilize every cut cell that is only a little bit smaller than background cells, meaning a value of $\alpha$ close to $1$ in~\eqref{eq: set of stab cut cells}. Instead we relax~\eqref{eq: standard cfl condition} by a constant factor $\frac{1}{4}$ to get a CFL condition
\begin{equation}\label{eq: cut cell cfl condition}
    \Delta t \leq \frac{1}{4}\frac{1}{2r + 1}\frac{h}{c}
\end{equation}
which still depends only on the background cell but does not require stabilization on large cut cells. For $\alpha$ in~\eqref{eq: set of stab cut cells} we choose $\alpha = \frac{1}{10}$ which in combination with~\eqref{eq: cut cell cfl condition} has worked out well in our numerical tests. 

For the cell stabilization terms introduced in definition~\ref{def: cell stabilization terms} we need to pick a set of propagation forms. We will always use the reflecting set of propagation forms defined in example~\ref{exam: reflecting propagation forms} which has all the properties we need for energy preservation and we will verify numerically that this is indeed the case.

To integrate the semidiscrete approximations in time, we use the following explicit strong-stability preserving (SSP) Runge-Kutta methods:
\begin{itemize}
    \item SSPRK(2,2): The two-stage, second-order method of Heun \cite{shu1988efficient}
    \item SSPRK(3,3): The three-stage, third-order method of Shu and Osher \cite{shu1988efficient}
    \item SSPRK(10,4): The ten-stage, fourth-order method of Ketcheson \cite{ketcheson2008highly}
\end{itemize}
Results on fully-discrete stability for linear problems for these methods have been obtained, e.g., in \cite{tadmor2002semidiscrete,ranocha2018L2stability,sun2019strong,tadmor2025stability}.

\subsection{Error measures}
We are interested in the behavior of the error of the discrete solution under mesh refinement of the background mesh, measured under various norms.
A natural choice for linear hyperbolic systems is the $L^2$ norm
\begin{equation}
    \left\|{e}\right\|_{L^2} = \left(\sum_{E\in\Mesh} \int_E {e}^2 dx \right)^{\frac 1 2}.
\end{equation}
As is often done for hyperbolic PDEs on cut-cell meshes we further investigate the error in the $L^\infty$-norm
\begin{equation}
\left\|e\right\|_{\infty} = 
    \sup\limits_{E\in{\Mesh}}
    \left\|e\right\|_{E,\infty}
    \approx
    \max\limits_{E\in{\Mesh}} \max\limits_{(x,\omega) \in \text{QR}(E)}|e(x)|
\end{equation}
which we approximate via a pointwise evaluation at the points of high-order quadrature rules $QR(E)$ on the mesh elements.
Due to their difference in size, it is to be expected that tiny cut cells induce numerical round-off errors that will be visible in the $L^\infty$-norm. To explore this issue, we introduce the filtered mesh
\begin{equation}
{\Mesh}_{\!,\rho} := \left\{ E \in \Mesh: \alpha_E > \rho \right\},
\end{equation}
and on this filtered mesh the filtered $L^\infty$-norm
\begin{equation} \label{eq: filtered inf norm}
    \left\|e\right\|_{\infty,\rho} = 
    \sup\limits_{E\in{\Mesh}_{\!,\rho}}
    \left\|e\right\|_{E,\infty}
    \approx
    \max\limits_{E\in{\Mesh}_{\!,\rho}} \max\limits_{(x,\omega) \in \text{QR}(E)} |e(x)|
\end{equation}
which computes an approximation of the pointwise error on all cells with a cell volume fraction larger than $\rho > 0$, i.e. it skips small cells during computation of the error. We want to point out again that the mesh ${\Mesh}_{\!, \rho}$ is only used to evaluate $\left\|e\right\|_{\infty,\rho}$. We never skip very small cells when computing a discrete solution.

The components of the exact solutions in our test cases will be denoted by $p$, $v_1$ and $v_2$, as above. Their discrete counterparts will be denoted by $p_h$, $v_{h, 1}$ and $v_{h, 2}$.

\subsection{Test 1: Rotated square}%
\begin{figure}[!htbp]
    \centering
    \begin{tikzpicture}[scale=2.5]
    \begin{scope}[shift={(0.58, 0)}]
    \draw[fill=gray] (-0.58, 0) -- (0.82, 0) -- (0.82, 1.4) -- (-0.58, 1.4) -- cycle;
    \end{scope}
    
    \begin{scope}[shift={(0.58, 0)}, rotate=35]
        \draw[fill= white] (0,0) -- (1, 0)  --(1, 1) -- (0, 1) -- cycle;
    \end{scope}
    
    \draw[step=0.1,gray] (0, 0) grid (1.4, 1.4);
    \end{tikzpicture}
    \caption{Test 1: Rotated square geometry and cut cell mesh example}
    \label{fig: rotated square}
\end{figure}
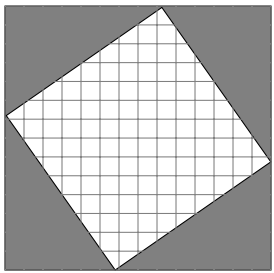
To verify the accuracy of our scheme we perform a convergence analysis on a rotated unit square that is placed inside a structured Cartesian grid, see Figure~\ref{fig: rotated square} for an illustration. On the unit square $\dhat\Omega = [0,1]^2$ our reference solution is given as
\begin{align*}
p(\dhat{x}, t) & = \sqrt{2} \pi [\sin(\sqrt{2} \pi t) - \cos(\sqrt{2} \pi t )] \cos(\pi \dhat x_1) \cos(\pi \dhat x_2),\\
v_1(\dhat x, t) & = -\pi [\cos(\sqrt{2}\pi t) + \sin(\sqrt{2}\pi t)] \sin(\pi \dhat x_1) \cos(\pi \dhat x_2),\\
v_2(\dhat x, t) & = -\pi [\cos(\sqrt{2}\pi t) + \sin(\sqrt{2}\pi t)] \cos(\pi \dhat x_1) \sin(\pi \dhat x_2)
\end{align*}
with $\dhat x \in \dhat \Omega$
and satisfies a reflecting wall boundary condition. By scaling and rotation the initial data given by the reference solution is then transformed onto the computational domain.
\begin{figure}[!htbp]
    \centering
    \includegraphics[width=\linewidth]{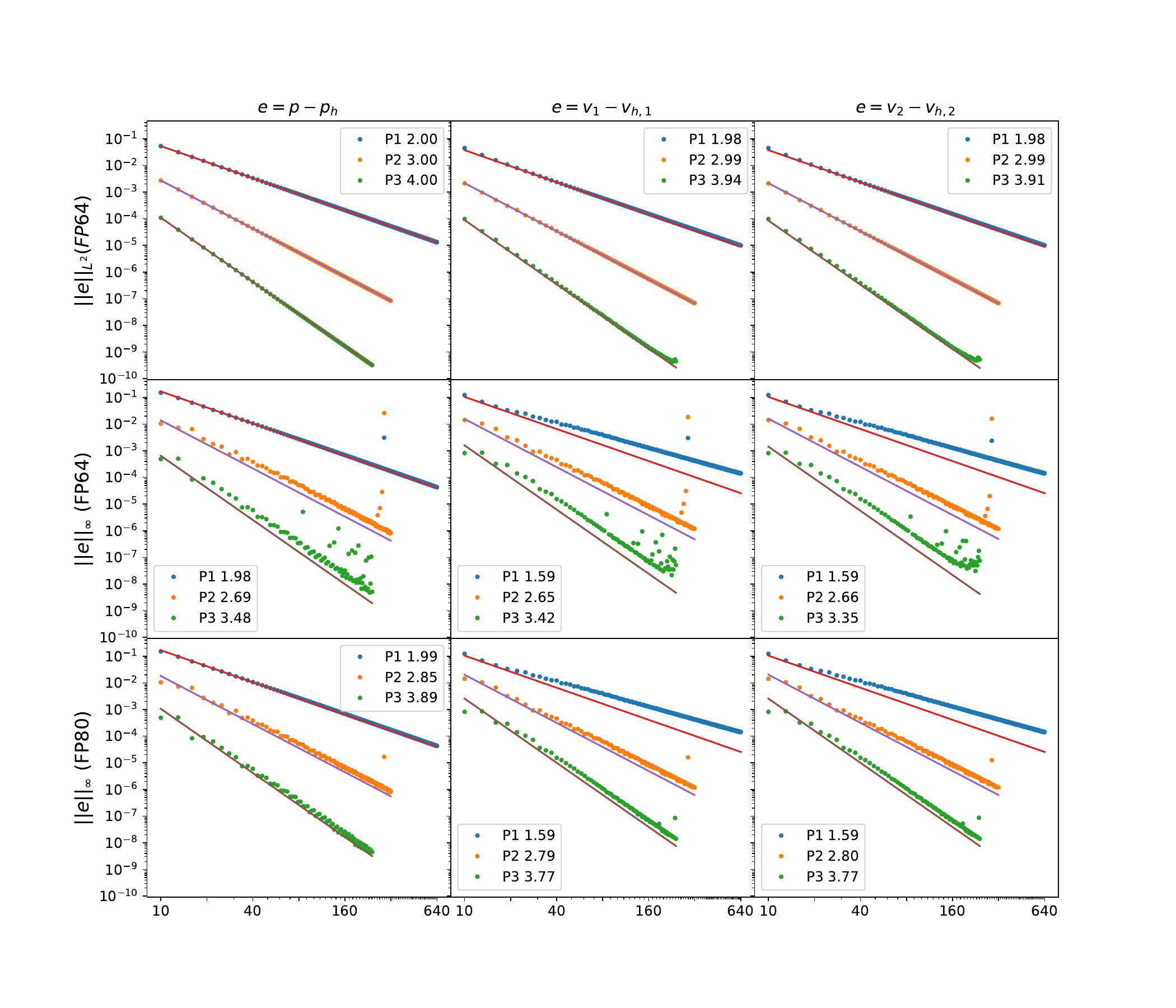}
    \caption{Numerical convergence results for polynomial degrees $1$, $2$ and $3$ on the rotated square. The three columns state the error in the pressure, first velocity and second velocity component, respectively. These errors are measured in both
    the $||\cdot||_{L^2}$ (first row) and $||\cdot||_{\infty}$ (second row) norm. In both cases the results where computed in double precision (FP64). While the convergence in the $||\cdot||_{L^2}$ norm shows the expected behavior, we see a reduced convergence rate and significant outliers in the $||\cdot||_{\infty}$. The last row again shows the error in the $||\cdot||_{\infty}$ but this time computed in extended double precision (FP80). Compared to the FP64 results, wee see a significant improvement, indicating that the outliers are caused by numerical precision issues.}
    \label{fig: rotated square results double}
\end{figure}

We compute the errors of the components in the $||\cdot||_{L^2}$ and $||\cdot||_{\infty}$ norm at the final time $T=1$ for polynomial degrees $r=1, 2, 3$, using the SSPRK schemes of corresponding order listed above.
The results are plotted in the first two rows of Figure~\ref{fig: rotated square results double} where the mesh resolution is described in terms of the number of background cells in one direction. Due to the way we construct our geometry, this leads to $h = \frac{1}{N}[\cos(\gamma) + \sin(\gamma)]$ where $h$ is the face length of a background cell and $\gamma$ the rotation angle ($35^{\circ}$ in this case).

We observe optimal convergence for the error in the $L^2$-norm for $p$ and $v$. In the $L^\infty$ norm the error shows a reduced convergence and (seemingly random) outliers.

Seeing a reduction of the convergence rate for cut-cell methods is not uncommon since the regular refinement process on the background mesh does not translate to a regular refinement on the cut cells, which can change drastically depending on the background mesh resolution.
To investigate whether this is a fundamental defect of our scheme or rather a numerical issue we rerun 
the test with a higher floating-point precision of 80 bit as opposed to the standard 64 bit of double precision. The results are compiled in the last row of figure~\ref{fig: rotated square results double} and look significantly better. Still there are some lesser issues to point out. The pointwise convergence rate for the velocity is still reduced, albeit mostly for $P^1$. We note that the rate of about $1.6$ is similar to the rates reported for the linear advection equation both for the DoD method in \cite{DoD_2d_linadv_2020, dod_icosahom} and for the SRD method in \cite{srd_dg}. Since these two methods are very different, this might point out a more fundamental issue for $P^1$ approximations on cut cells. For $P^2$ and $P^3$ the situation seems better, even though our rates for the wave equation are slightly worse than what is reported in \cite{dod_icosahom} and \cite{srd_dg}. This might still be caused by numerical inaccuracies, as indicated by remaining outliers in our data plots.

To investigate the issue further, we consider a $289 \times 289$ mesh, which contains extremely small cut cells, and a $P^2$ polynomial approximation. The aim is to investigate whether the significant increase in the (pointwise) error is restricted to (very) small cells or whether it propagates into the global domain. To do so we measure the error in the filtered $L^\infty$-norm $\|\cdot\|_{\infty,\rho}$ as defined in \eqref{eq: filtered inf norm}. As filter threshold $\rho$ we use values corresponding roughly to the different volume fractions occurring in the cut-cell mesh. The results are compiled in Table~\ref{tab: very small cell errors}. We see that for this mesh configuration the pointwise error deteriorates severely only on the very small cells with a volume fraction of about $1e^{-12}$ (and absolute volume of about $1e^{-16}$). We conclude that the increase in the error on the very small cells does not propagate into the domain or even into the neighboring cells, indicating the stability of our scheme.

We have to state here that the handling of cells that have a volume close to machine precision is a delicate issue. Often these very small cells are simply skipped during the meshing process and subsequently during computation. We do not want to make an attempt here on deciding which cells should be skipped but rather only point out that even in the presence of such very small cells the implementation of our scheme seems to stay stable and preserve accuracy around these cells.

\begin{table}[!htbp]
    \centering
    \begin{tabular}{c||D{x}{\cdot}{-1}D{x}{\cdot}{-1}D{x}{\cdot}{-1}}
         Threshold $\rho$
         & \multicolumn{1}{c}{$\left\|p_h-p\right\|_{\infty,\rho}$}
         & \multicolumn{1}{c}{$\left\|v_{h,1}-v_1\right\|_{\infty,\rho}$}
         & \multicolumn{1}{c}{$\left\|v_{h,2}-v_2\right\|_{\infty,\rho}$} \\
         \hline
         $10^{-12}$ & 2.60128x10^{-2} & 1.83481x10^{-2} & 1.5996x10^{-2} \\
         $10^{-7}$ & 1.25392x10^{-6} & 1.67888x10^{-6} & 1.67955x10^{-6} \\
         $10^{-5}$ & 1.25392x10^{-6} & 1.67888x10^{-6} & 1.67955x10^{-6} \\
         $10^{-4}$ & 9.90338x10^{-7} & 1.67888x10^{-6} & 1.67955x10^{-6} \\
         $10^{-2}$ & 7.75344x10^{-7} & 1.67888x10^{-6} & 1.67955x10^{-6} \\
         $10^{-1}$ & 4.97874x10^{-7} & 1.67888x10^{-6} & 1.67955x10^{-6}
    \end{tabular}
    \caption{Errors in the filtered $L^\infty$-norm \eqref{eq: filtered inf norm} for the rotated square with $289 \times 289$ background elements and different values for the threshold $\rho$. These thresholds correspond roughly to volume fractions $\alpha_E$ of cut cells occuring in the mesh. If we skip the tiny cells with $\alpha_E \approx 10^{-12}$ (and an absolute volume of about $10^{-16}$ on this mesh) we see a strong improvement of the error.}
    \label{tab: very small cell errors}
\end{table}

\FloatBarrier
\subsection{Test 2: Rotated channel}
In this second test we investigate the long term behavior of the stabilized method.
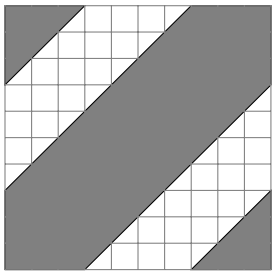
\begin{figure}[!htbp]
    \centering
    \begin{tikzpicture}[scale=3.5]

    \draw[fill=gray] (0.0, 1.0) -- (0.3001, 1.0) -- (0.0, 0.6999) -- cycle;

    \draw[fill=gray] (0.0, 0.0) -- (0.0, 0.2999) -- (0.7001, 1.0) -- (1.0, 1.0) -- (1.0, 0.6999) -- (0.3001, 0.0) -- cycle;

    \draw[fill=gray] (1.0, 0.0) -- (1.0, 0.2999) -- (0.7001, 0.0) -- cycle;
    
    \draw[step=0.1,gray] (0, 0) grid (1.0, 1.0);
    \end{tikzpicture}
    \caption{Rotated channel geometry and cut-cell mesh example}
    \label{fig: rotated channel}
\end{figure}
\begin{figure}[!htbp]
    \centering
    \includegraphics[width=\linewidth]{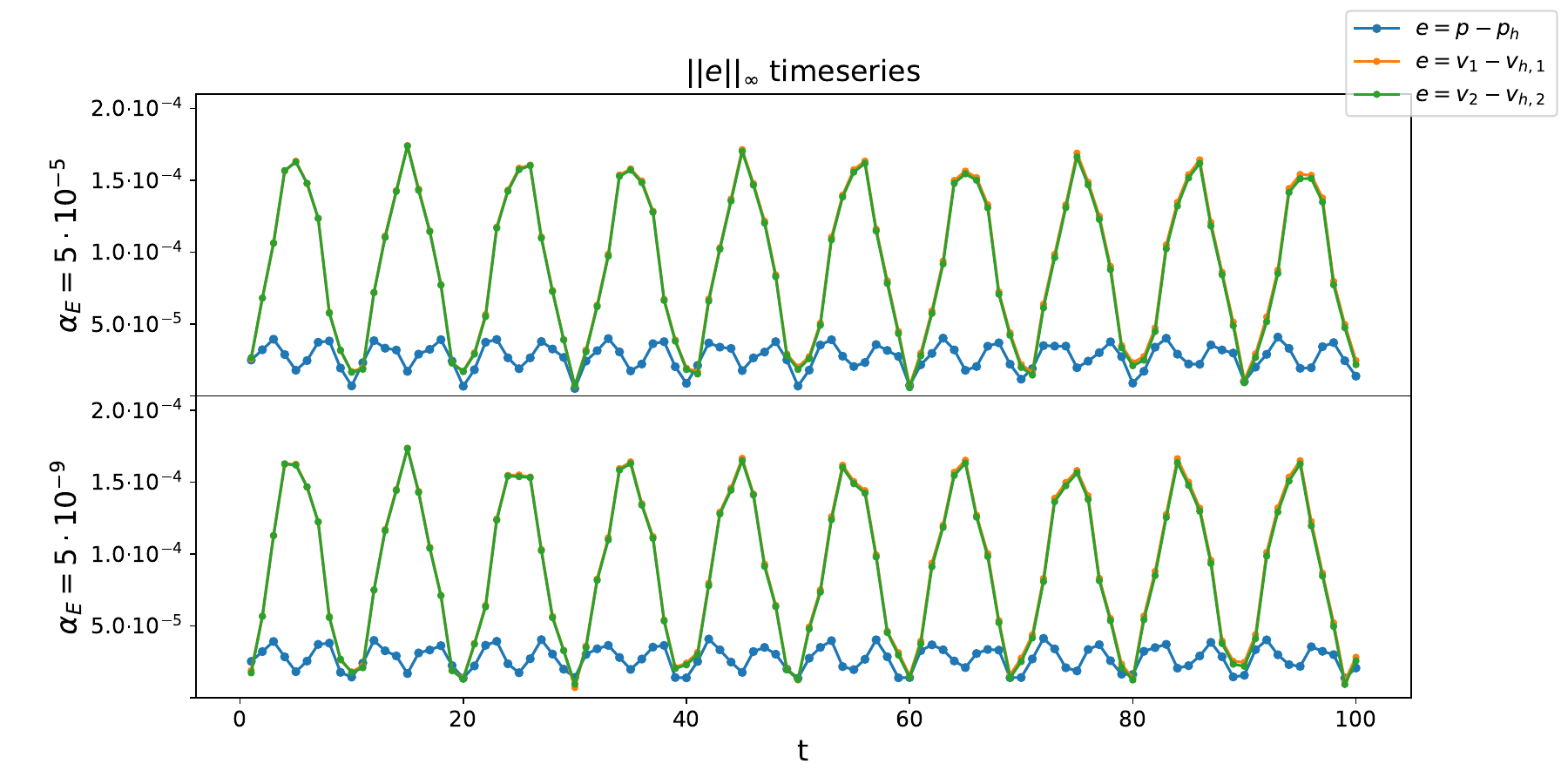}
    \caption{Behavior of the error in the $||\cdot||_{\infty}$ norm for two configurations of the channel test in the time interval $[0, T]$. The two configurations have small cells with volume fraction $\alpha_E = 5\cdot10^{-5}$ (top row) and $\alpha_E = 5\cdot10^{-9}$ (bottom row). We can see some oscillatory behavior, especially for the velocity errors but overall all errors stay bounded.}
    \label{fig: rotated channel error timeseries}
\end{figure}
\begin{table}[!htbp]
    \centering
    \begin{tabular}{c c|c|c}
         & \multicolumn{3}{c}{Total energy $E_h(t)$}\\
         & Time & $\alpha_E = 10^{-5}$ & $\alpha_E = 10^{-9}$ \\
         \hhline{~ =|=|=}
         \multirow{3}{*}{No dissipation} & $t = 0$ & 1.199999999981807 & 1.199999999981823 \\
         & $t = 20$ & 1.199999999981692 & 1.199999999981697   \\
         & $t = 40$ & 1.199999999981597 & 1.199999999981622 \\
         & $t = 60$ & 1.199999999981529 & 1.199999999981551 \\
         & $t = 80$ & 1.199999999981413 & 1.199999999981459 \\
         & $t = 100$ & 1.19999999998137 & 1.199999999981395 \\
         \cline{2-4}
         \multirow{3}{*}{With dissipation} & $t = 0$ & 1.199999999981807 & 1.199999999981823 \\
         & $t = 20$ & 1.199999443849309 & 1.199999443921604   \\
         & $t = 40$ & 1.199998887680314 & 1.199998887824876 \\
         & $t = 60$ & 1.199998331511058 & 1.199998331727847 \\
         & $t = 80$ & 1.199997775342008 & 1.199997775631023 \\
         & $t = 100$ & 1.199997219173264 & 1.199997219534485
    \end{tabular}
    \caption{Total energy $E_h(t)$ values after specific times for the rotated channel test. The left column features energy values on a channel where small cells have a volume fraction of $\alpha_E = 10^{-5}$. The right columns does so for a channel with small cell volume fraction of $\alpha_E = 10^{-9}$. In the first row we have results in case of no added dissipation, i.e. $S_n = 0$. Here we see essentially no change in the total energy. The second row displays results if $S_n$ is the classical Lax-Friedrichs dissipation. We can see the expected drop in the total energy.}
    \label{tab: total energy channel}
\end{table}
Here we want to investigate the long term behavior of our scheme. We consider as a background mesh a unit square with periodic boundary conditions and introduce cuts at a $45^{\circ}$ angle inside the square to generate a periodic channel as depicted in Figure~\ref{fig: rotated channel}. Along the cut interfaces we specify a reflecting wall boundary condition. The result is a channel on the torus through which we send a plane wave. The initial data is given as
\begin{align*}
p(\dhat{x}, t) & = \sin(2 \pi [\dhat{x}_1 - t]),\\
v_1(\dhat{x}, t) & = \sin(2 \pi [\dhat{x}_1 - t]),\\
v_2(\dhat{x}, t) & = 0
\end{align*}
on the unit square and then transformed onto the computational domain. The channel length is $l = \sqrt{2}$, corresponding to a cut angle of $45^{\circ}$. As final time we choose $T=100\cdot l$. Further we set $S_n = 0$, meaning that we do not add dissipation. For the discretization we choose a $P^2$ discrete space and the standard explicit SSP-RK3 scheme. Two tests are run on a resolution of $100$ background cells in each direction, with offsets for the channel such that we have a volume fraction for the small cells of $\alpha_E = 10^{-5}$ and $\alpha_E = 10^{-9}$, respectively. Figure~\ref{fig: rotated channel error timeseries} shows plots for the pointwise error at different times for both mesh configurations. We see some oscillations in the error, particularly for the velocity, but overall the error stays bounded.

To display the energy conservation and stability properties of our scheme, we repeat the test on a mesh with 50 background cells and a fourth order SSP-RK scheme so that the effect of the time discretization error is minimized. We consider again $S_n = 0$, i.e., no dissipation, and in addition a Lax-Friedrichs dissipation. Table~\ref{tab: total energy channel} contains the total energy values
\begin{equation}
E_h(t) = ||u_h(\cdot, t)||_{L^2}
\end{equation}
at certain points in time for both mesh configurations. The upper half shows the energy values for a conservative discretization (i.e. $S_n = 0$) while the lower half does so for a Lax-Friedrichs type dissipation. In the conservative case we see very small changes in the energy even for $T=100$ while for the dissipative case we can see significant dissipation of energy.

% don't put figures after the conclusions
\FloatBarrier

\section{Conclusions}
In this paper we have presented a generalized construction of a DoD stabilization for cut-cell DG discretizations to support linear systems of first order hyperbolic PDEs, in particular the linear wave equation. The construction is based on propagation forms that mimic the transport across small cut cells and fulfill an integration-by-parts type consistency. With this we proved energy preservation of the method, which is also backed by our numerical tests. The method shows optimal convergence and is robust in long-time simulations. The generalized construction explicitly distinguishes between central fluxes and diffusive stabilization and by this is able to control the energy error and even construct a fully energy conservative discretization.

In future work we will extend this method to non-linear settings, in particular the Euler equations. Also the abstract construction via propagation forms should ease the path to a robust stabilization in 3D.
\subsection*{Funding}
The research of GB, CE, LP, and HR was supported by the Deutsche Forschungsgemeinschaft (DFG, German Research Foundation) - SPP 2410 Hyperbolic Balance Laws in
Fluid Mechanics: Complexity, Scales, Randomness (CoScaRa), specifically within the project 526031774 (EsCUT: Entropy-stable high-order CUT-cell discontinuous Galerkin methods).
SM, CE and GB gratefully acknowledge support by the Deutsche Forschungsgemeinschaft (DFG, German Research Foundation) - 439956613 (Hypercut).
CE and GB further acknowledge support by the Deutsche Forschungsgemeinschaft (DFG, German Research Foundation) under Germany's Excellence Strategy EXC 2044 –390685587, and EXC 2044/2 –390685587, Mathematics Münster: Dynamics–Geometry–Structure.

\printbibliography%{biblio}
%\bibliographystyle{plain}
%\bibliography{biblio}

\end{document}